\documentclass[11pt,reqno]{amsart}

\usepackage{geometry} 
\usepackage{comment} 
\usepackage{caption}
\usepackage{graphicx}

\DeclareGraphicsExtensions{.png,.pdf,.jpg} 

\usepackage[english]{babel} 
\addto\captionsenglish{

}
 
\usepackage{hyperref}
\usepackage{amsmath}
\usepackage{mathrsfs}
\usepackage{xcolor}

\numberwithin{equation}{section}
\newcommand{\normmm}[1]{{\left\vert\kern-0.25ex\left\vert\kern-0.25ex\left\vert #1
    \right\vert\kern-0.25ex\right\vert\kern-0.25ex\right\vert}}

\usepackage{fancyhdr}
\usepackage{amsthm} 

\begin{document}
\newtheorem{Proposition}{Proposition}[section]
\newtheorem{Lemma}[Proposition]{Lemma}
\newtheorem{Corollary}[Proposition]{Corollary}
\newtheorem{Definition}[Proposition]{Definition}
\newtheorem{Theorem}[Proposition]{Theorem}
\newtheorem{Remark}[Proposition]{Remark}
 



\vspace{8mm}

\title[Inverse obstacle scattering with a single moving emitter]{Inverse obstacle scattering with a single moving emitter}
\author[Yu Sun et al]{Yu Sun, Bo Chen$^*$, Peng Gao, Qiuyi Li and Yao Sun$^*$}

\thanks{College of Science, Civil Aviation University of China, Tianjin, People's Republic of China}
\thanks{Corresponding author: Bo Chen: charliecb@163.com; Yao Sun: sunyao10@mails.jlu.edu.cn}
\thanks{Submitted October 10, 2025.}

\begin{abstract}
This paper is concerned with time domain forward scattering and inverse scattering problems with a single moving point source as the emitter. Approximate solutions are provided for the forward scattering problem with a moving emitter. Regarding the inverse problem, in addition to a basic indicator function based on the approximate solutions, a novel indicator function is developed to construct the direct sampling method to recover both point-like and extended scatterers. Numerical experiments demonstrate that the proposed algorithms are effective in reconstructing both two-dimensional and three-dimensional scatterers with a single moving emitter.
\end{abstract}

\maketitle





\section{Introduction}

Scattering and inverse scattering problems have achieved considerable developments over the past decades. In general, the scattering problems can be divided into two categories, frequency domain problems and time domain problems. For frequency domain problems, substantial theoretical achievements have been published \cite{BG scattering elastic shell medium,D Colton book,Integral Equation Methods,Y Sun Indirect boundary integral equation method}, and diverse numerical algorithms have been developed, such as factorization methods \cite{K. H. Leem factorization methods inverse scattering}, iterative methods \cite{Gang Bao iteration method}, sampling methods \cite{D Colton LSM inverse scattering,JUN ZOU near field,X. Liu sampling method for multiple multiscale targets}, and other imaging methods \cite{G. Bao Inverse scattering direct imaging algorithm, J. Lai direct imaging method for the inverse obstacle scattering, J. Li singular and hypersingular}. Additionally, various related problems are studied, such as the co-inversion of sources and scatterers \cite{Y. Chang Simultaneous recovery near-field,D. Zhang Co-inversion}. 

Compared with frequency domain problems, the study of time domain problems is more complex. However, time domain problems have received more and more attentions since time-dependent data is generally more accessible. Owing to advancements in various technologies such as the time domain boundary integral equation method \cite{CB scattering and inverse scattering,Sayas Time Domain Boundary Integral Equations} and the rapid solution of wave equation \cite{Banjai Rapid solution}, the time domain direct scattering problem has been well analyzed and numerically solved. For time domain inverse scattering problems, sampling methods, including the linear sampling method and direct sampling methods, are important among various numerical methods. The time domain linear sampling method is proved to be effective for Dirichlet obstacles \cite{Q Chen LSM effective}, Neumann and Robin obstacles \cite{Haddar improved time domain linear sampling for Robin and Neumann}, penetrable medium \cite{GYK Toward linear sampling method} and cracks \cite{Y Yue linear sampling method for cracks}. Meanwhile, the research on time domain direct sampling methods has also made significant progress. The advantage of this type of method lies in the fact that it does not require solving the time-dependent wave equation. Effective and efficient sampling schemes based on the migration method are proposed in \cite{GYK DSM reverse time}, and the numerical algorithms based on the Green's function and time convolution are proposed in \cite{CB point-like scatterer} and \cite{YQQ DSM Green function}. Based on the Fourier-Laplace transform, an inherent connection between the time domain and frequency domain direct sampling methods is established in \cite{GYK DSM connection}. The asymptotic properties of the indicator function for the point-like scatterer are proved according to the Fourier-Laplace transform in \cite{Hong Guo DSM}.

In this paper, we consider the forward and inverse obstacle scattering with a single moving point source as the emitter. Inverse scattering with moving point sources has engineering applications including geoacoustic inversion \cite{Yining Shen Geoacoustic inversion} and Doppler tomography imaging \cite{Doppler Tomography Imaging}. For typical forward and inverse scattering problems, the incident wave is usually chosen to be a spherical wave or a plane wave. Few studies concern with incident waves from moving point sources. Nevertheless, the research of the inverse source problem with moving point sources is helpful for the analysis in this paper. For instance, \cite{algebraic moving ui} provides an algebraic expression for the wave field generated by a moving point source, and \cite{WXC} investigates the mathematical design of moving point sources as acoustic emitters. Additionally, numerical schemes to reconstruct three-dimensional time-dependent point sources are proposed in \cite{CB reconstruct three dimensional time-dependent point sources}, and a deterministic-statistical approach for reconstructing the moving path of a moving source is proposed in \cite{GYK statistical reconstruct moving sources}.

For the forward scattering problem with a moving point source, we derive approximate solutions to the scattering problem under the assumption that the speed of the moving source is below the sound speed. For the inverse problem, we propose an indicator function that is directly dependent on the approximate solution, as well as a novel indicator function based on the time convolution. Numerical experiments with a single moving emitter validate the efficiency of the direct sampling methods based on the proposed indicator functions, and both point-like and extended scatterers can be accurately reconstructed. 

The rest of the paper is organized as follows. In Section 2, we present the analysis of the forward scattering problem, and the corresponding approximate solutions are exhibited. Section 3 is devoted to the construction and analysis of time domain direct sampling methods. Numerical experiments are provided in Section 4 to verify the effectiveness of our methods. The last section is the conclusion.

\section{Forward scattering problem and the approximate solutions}

Considering the forward scattering problem with a moving emitter, let $D \subset \mathbb{R}^3$ be the region where the sound-soft scatterer is located. Assume that $D$ is bounded with piecewise Lipschitz boundary $\partial D $ and connected exterior $\mathbb{R}^{3} \setminus \overline{D}$. Denote by $s(t)$ the trajectory of the moving emitter, where $s:\mathbb{R} \to \mathbb{R}^3$ is assumed to be $C^1$-continuous. Suppose that the measurement surface $\Gamma_m \subset  \mathbb{R}^{3} \setminus \overline{D}$ is a closed Lipschitz surface. Figure \ref{geo} presents the two-dimensional geometrical setting of the time-domain acoustic scattering problem with a single moving emitter.

\begin{figure}[!htbp]
		\centering
		\begin{tabular}{ccc}
			\includegraphics[width=0.4\textwidth]{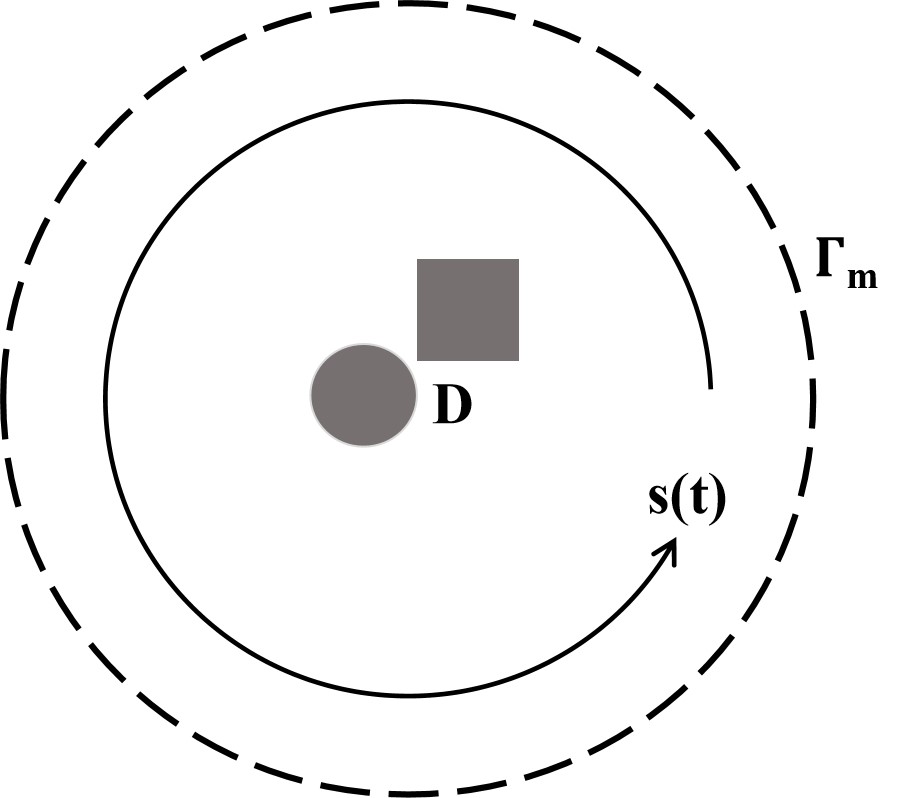}
		\end{tabular}
\caption{Two-dimensional geometrical setting of the problem. }
\label{geo}
\end{figure}

Denote by $u^i$  the incident field, $u^{tot}$ the total field, and $u=u^{tot}-u^i$ the scattered field. The incident wave field $u^i(x,t)$ generated by a single moving point source satisfies 
\begin{equation*}
c^{-2}\partial_{tt}u^i(x,t)-\triangle u^i(x,t)=\lambda(t)\delta(x-s(t)), \quad \quad x\in \mathbb{R}^{3},~t\in \mathbb{R},
\end{equation*}
where $c>0$ denotes the sound speed of the homogeneous background medium, $\partial_{tt}u^i=\frac{\partial^2u^i}{\partial t^2}$, $\triangle$ is the Laplacian in $\mathbb{R}^{3}$, $\lambda(t)$ is the signal function, and $\delta$ is the Dirac delta distribution. Additionally, the signal function $\lambda(t)$ is assumed to be causal, which means $\lambda(t)=0$ for $t<0$.

Denote by $f(x,t)$ the source term, that is, $f(x,t)=\lambda(t)\delta (x-s(t))$. Then $u^{i}(x,t)$ can be represented as 
\begin{equation*}
    u^{i}(x,t)=\int_{\mathbb{R}} \int_{\mathbb{R}^{3} }G(x,t-\mu;y)f(y,\mu) \mathrm{d}y\mathrm{d}\mu,
\end{equation*}
where
\begin{equation*}
    G(x,t;y)=\frac{\delta(t-c^{-1}|x-y|)}{4\pi|x-y|}
\end{equation*}
is the Green's function for the operator $c^{-2}\partial_{tt}-\bigtriangleup$ in $\mathbb{R}^{3}$.

Assume that the moving velocity $v(t)=\frac{\mathrm{d}s(t)}{\mathrm{d}t}$ of the moving emitter satisfies $\left |v(t)  \right | <c$. Then the incident field generated by the moving point source is \cite{algebraic moving ui}
\begin{equation}\label{eq:ui for moving point source}
u^{i}(x,t;s(t))=\frac{\lambda(\tau)}
{4\pi|x-s(\tau)|\left(1-\frac{v(\tau)\cdot(x-s(\tau))}{c|x-s(\tau)|}\right)}, \quad \quad x \in \mathbb{R}^{3} \backslash \{s(t)\},~t \in \mathbb{R},
\end{equation}
where the retarded time $\tau$ satisfies
\begin{equation*}
t=\tau+c^{-1}|x-s(\tau)|.
\end{equation*}

The scattered field $u(x,t)$ satisfies the homogeneous wave equation
\begin{equation}\label{eq:problem1}
c^{-2}\partial_{tt}u(x,t)-\triangle u(x,t)=0, \quad \quad x\in \mathbb{R}^{3},~t\in \mathbb{R},
\end{equation}
the Dirichlet boundary condition
\begin{equation}\label{eq:problem2}
u(x,t)=-u^i(x,t), \quad \quad x\in \partial D,~t\in \mathbb{R}
\end{equation}
and the initial condition
\begin{equation}\label{eq:problem3}
u(x, 0)=\partial_{t}u(x, 0)=0, \quad \quad x\in \mathbb{R}^3 \setminus \overline{D}.
\end{equation}

The Kirchhoff's formula indicates that the solution to the wave equation \eqref{eq:problem1} can be represented as \cite{Sayas Time Domain Boundary Integral Equations}
\begin{equation}
u(x,t)=\int_{\partial D} (G*g )(x,t;y)\mathrm{d}s_y, \quad \quad x\in \mathbb{R}^3 \backslash \partial D,~t \in \mathbb{R},
\label{Kirchhoff}
\end{equation}
where 
\begin{equation*}
    (G*g )(x,t;y) = \int_{\mathbb{R}} G(x,t-\mu;y)g(\mu;y) \mathrm{d}\mu=\frac{g(t-c^{-1}|x-y|;y)}{4\pi|x-y|},
\end{equation*}
and $g(t;y)$ is the solution to the retarded potential boundary integral equation
\begin{equation}
\int_{\partial D} (G*g )(x,t;y)\mathrm{d}s_y=-u^i(x,t), \quad \quad x\in \partial D,~t \in \mathbb{R}.
\label{eq:boundary integral equation}
\end{equation}

The approximate solution to the direct scattering problem is an important tool for the reconstruction of the scatterer. We provide the analysis of the approximate solutions in the following. 

\begin{Lemma}[cf.~Bo Chen \& Yao Sun~\cite{CB point-like scatterer}]\label{Lemma:approximate solution}
Let  $u(x, t)$  be the solution to the forward scattering problem \eqref{eq:problem1}-\eqref{eq:problem3} and  $g(t ; y)$  be the solution to the boundary integral equation \eqref{eq:boundary integral equation}. Assume that  $g(t ; y)$  is Lipschitz continuous with respect to both  $t \in \mathbb{R}$  and  $y \in \partial D $. Moreover, assume that $diam  (D)=   \varepsilon(x) d(x)$  for $ x \in \Gamma_{m} $, where  $0<\varepsilon(x) \ll 1 $ and  $d(x)=\underset{ y \in \partial D}{\min} |x-y|$ . For any  $y_{0} \in D$  and  $y_{1} \in \partial D $ , we have
\begin{equation*}
    u(x, t)=\frac{Ag(t-c^{-1}|x-y_0|;y_1)}{4\pi|x-y_0|}+O(\varepsilon(x)), \quad \quad x \in \Gamma_{m}, t \in \mathbb{R},
\end{equation*}
where $A$ is the area of  $\partial D$ .
 \end{Lemma}

\begin{Proposition}\label{Proposition: approximate solution}
 Let $u(x, t)$ be the solution to the forward scattering problem \eqref{eq:problem1}-\eqref{eq:problem3}, and  $g(t ; y)$  be the solution to the boundary integral equation \eqref{eq:boundary integral equation}. Assume that $g(t ; y)$  is Lipschitz continuous with respect to both  $t \in \mathbb{R}$  and  $y \in \partial D $. The incident wave field \eqref{eq:ui for moving point source} is generated by a single moving point source with the smooth trajectory $s(t)$ and the Lipschitz continuous signal function $\lambda(t)$. Assume that  $D \subset \mathbb{R}^{3}$ is a Lipschitz domain which is separated from both $\{s(t)\}$ and the measurement surface $\Gamma_m$, $y_0$ is the geometric center of the region $D$, $diam  (D)=   \varepsilon(x) d(x)$  for $ x \in \Gamma_{m} $ and $diam (D)= \varepsilon(s(t))d(s(t))$ for $t\in \mathbb{R}$, where $0<\varepsilon(x)\ll 1 $,  $0<\varepsilon(s(t))\ll 1 $,  $d(x)=\underset{ y \in \partial D}{\min} |x-y|$ and $d(s(t))=\underset{ y \in \partial D}{\min} |s(t)-y|$. \\
\textup{(1)}~Assume that $|v(t)|<c$, we have
\begin{equation*}
    u(x,t)=-C\frac{\lambda(\tilde{t}-c^{-1}|s(\tilde{\tau})-y_0|)}
{4\pi|x-y_0||s(\tilde{\tau})-y_0|\left(1-\frac{v(\tilde{\tau})\cdot(y_0-s(\tilde{\tau}))}{c|y_0-s(\tilde{\tau})|}\right)}+O(\varepsilon(x, t)),  \quad \quad x \in \Gamma_{m},~t \in \mathbb{R},
\end{equation*}
where $\varepsilon(x, t)=\max \left \{ \varepsilon (x),  \varepsilon(s(\tilde{t})),  \varepsilon(s(\tilde{\tau} )) \right \} $, $\tilde{t}=t-c^{-1}|x-y_0|$, $\tilde{\tau}$ satisfies $\tilde{t}=\tilde{\tau}+c^{-1}|y_0-s(\tilde{\tau})|$, and $C$ is a constant depending only on $y$ and $D$. \\
\textup{(2)}~Moreover, if $|v(t)|= \varepsilon(v(t))c$ with $0<\varepsilon(v(t))\ll 1$ for $t\in \mathbb{R}$, we have
\begin{equation*}
    u(x,t)=-C\frac{\lambda(\tilde{t}-c^{-1}|s(\tilde{\tau})-y_0|)}
{4\pi|x-y_0||s(\tilde{\tau})-y_0|}+O(\varepsilon'(x, t)),  \quad \quad x \in \Gamma_{m},~t \in \mathbb{R},
\end{equation*}
where $\varepsilon'(x, t)=\max \left \{ \varepsilon(x, t),  \varepsilon(v(\tilde{\tau})) \right \} .$
\end{Proposition}

\begin{proof}
\textup{(1)}~For the geometric center $y_{0}$ of the scatterer $D$  and any $y_{1} \in \partial D$, Lemma \ref{Lemma:approximate solution} implies that the solution to the forward scattering problem \eqref{eq:problem1}-\eqref{eq:problem3} is
\begin{equation*}
    u(x, t)=\frac{Ag(t-c^{-1}|x-y_0|;y_1)}{4\pi|x-y_0|}+O(\varepsilon(x)), \quad \quad x \in \Gamma_{m}, t \in \mathbb{R},
\end{equation*}
where $A$ is the area of  $\partial D $, and  $g(t ; \tilde{y}) $ is the solution of the boundary integral equation
\begin{equation}
\int_{\partial D} \frac{g\left(t-c^{-1}\left|y-\tilde{y}\right|; \tilde{y}\right)}{4 \pi\left|y-\tilde{y}\right|} \mathrm{d} s_{\tilde{y}}
=-\frac{\lambda(t-c^{-1}|y-s(\tau)|)}
{4\pi|y-s(\tau)|\left(1-\frac{v(\tau)\cdot(y-s(\tau))}{c|y-s(\tau)|}\right)}.
\label{eq:BIE apply}
\end{equation}

For any  $y$, $\tilde{y}$, $y_{1} \in \partial D$, the fact that $g(t ; \tilde{y}) $ is Lipschitz continuous with respect to both  $t \in \mathbb{R}$  and  $\tilde{y} \in \partial D $ implies
\begin{equation*}
\begin{aligned}
|g(t-c^{-1}|y-\tilde{y}|; \tilde{y}) - g(t; y_{1})| &\leq |g(t-c^{-1}|y-\tilde{y}|; \tilde{y}) - g(t-c^{-1}|y-\tilde{y}|; y_{1})|\\
& \quad + |g(t-c^{-1}|y-\tilde{y}|; y_{1}) - g(t; y_{1})|  \\
&\leq K_{1}|\tilde{y}-y_{1}|+c^{-1}K_{2}|y-\tilde{y}| \\
&\leq (K_{1}+c^{-1}K_{2})diam(D) 
\end{aligned}
\end{equation*}
for some $K_1, K_2 > 0$, that is
\begin{equation}
    g\left(t-c^{-1}\left|y-\tilde{y}\right| ; \tilde{y}\right)=g\left(t ; y_{1}\right)+O(\varepsilon(s(t))).
\label{g(t;y)}
\end{equation}

From the assumption that $D$ and $\left \{ s(t) \right \} $ are separated, we can get $|y_0-s(\tau)|\ne 0$. Together with the fact that
$$|y-s(\tau )| = |y_0-s(\tau )| + O(\varepsilon(s(\tau )))$$
and $|v(t)|<c$, we have
\begin{equation}
\frac{1}{4\pi|y-s(\tau )|}=\frac{1}{4\pi|y_0-s(\tau)|} +O(\varepsilon(s(\tau ))
\label{approx1}
\end{equation}
and
\begin{equation}
\frac{1}{1-\frac{v(\tau)\cdot(y-s(\tau))}{c|y-s(\tau)|}}=\frac{1}{1-\frac{v(\tau)\cdot(y_0-s(\tau))}{c|y_0-s(\tau)|}}+O(\varepsilon(s(\tau ))).
\label{approx2}
\end{equation}
Combining \eqref{approx1}, \eqref{approx2} with the fact that $\lambda(t)$ is Lipschitz continuous, the following conclusion can be obtained:
\begin{equation*}
    \frac{\lambda(t-c^{-1}|y-s(\tau)|)}
{4\pi|y-s(\tau)|\left(1-\frac{v(\tau)\cdot(y-s(\tau))}{c|y-s(\tau)|}\right)}=\frac{\lambda(t-c^{-1}|y_0-s(\tau)|)}
{4\pi|y_0-s(\tau)|\left(1-\frac{v(\tau)\cdot(y_0-s(\tau))}{c|y_0-s(\tau)|}\right)}+O(\varepsilon(s(\tau ))).
\end{equation*}

Furthermore, combining the above equation with equation  \eqref{g(t;y)}, the boundary integral equation \eqref{eq:BIE apply} can be deduced to
\begin{equation}
g\left(t ; y_{1}\right) \int_{\partial D} \frac{1}{4 \pi\left|y-\tilde{y}\right|} \mathrm{d} s_{\tilde{y}} 
 = -\frac{\lambda\left(t-c^{-1}\left|y_{0}-s(\tau)\right|\right)}{4 \pi\left|y_{0}-s(\tau)\right|\left(1-\frac{v(\tau)\cdot(y_0-s(\tau))}{c|y_0-s(\tau)|}\right)}+O(\varepsilon_1(t)),   
\label{g(t;y_1)}
\end{equation}
where $\varepsilon_1(t)=\max \left \{ \varepsilon (s(t)),  \varepsilon(s(\tau )) \right \} $. It is worth mention that $\tau$ is unique determined by $t$ with $t=\tau+c^{-1}|y_0-s(\tau)|$, and $\varepsilon_1(t)$ depends only on $t$ \cite{algebraic moving ui}.

Since $D$ is a Lipschitz domain, for any $y\in \partial D$, there exists a spherical neighborhood $B_\rho(y)$, and a bijection $F_y: B_\rho(y)\to B_1(0):=\left \{x\in \mathbb{R}^3: |x|\leq 1 \right \}$ satisfies
\begin{equation*}
    F_y(\partial D\cap B_\rho(y))= P_0,
\end{equation*}
where $P_0=\left \{x=(x_1, x_2, x_3)\in B_1(0): x_3=0 \right \} $, and that $F_y$, $F_y^{-1}$ are Lipschitz continuous with the Lipschitz constants $L_1$ and $L_2$, respectively. Then, for any $\tilde{y}\in \partial D\cap B_\rho(y)$, denote $x=F_y(y)$, $\tilde{x}=F_y(\tilde{y})$, it holds that
\begin{equation*}
    \frac{1}{L_1}|x-\tilde{x}|\le |y-\tilde{y}|\le L_2|x-\tilde{x}|.
\end{equation*}
Therefore, given that the area of $\partial D$ is $A$, one can verify that
\begin{equation*}
\begin{aligned}
\int_{\partial D} \frac{1}{4 \pi\left|y-\tilde{y}\right|} \mathrm{d} s_{\tilde{y}}&=\int_{\partial D\cap B_\rho(y)} \frac{1}{4 \pi\left|y-\tilde{y}\right|} \mathrm{d} s_{\tilde{y}} + \int_{\partial D \setminus  B_\rho(y)} \frac{1}{4 \pi\left|y-\tilde{y}\right|} \mathrm{d} s_{\tilde{y}}\\
&\leq \int_{P_0} \frac{L_1 L_2^2}{4 \pi\left|x-\tilde{x}\right|} \mathrm{d} s_{\tilde{x}} + \frac{A}{4\pi \rho}\\
&= \frac{L_1 L_2^2}{4 \pi}\int_{P_0} \frac{1}{\left|x-\tilde{x}\right|} \mathrm{d} s_{\tilde{x}} + \frac{A}{4\pi \rho}\\
&\leq \frac{L_1 L_2^2}{2} + \frac{A}{4\pi \rho}.
\end{aligned}
\end{equation*}
The above equation implies that the generalized integral converges. Therefore,
there exists a constant  $E$  depending only on  $y$  and  $D$  such that
\begin{equation*}
    \int_{\partial D} \frac{1}{4 \pi\left|y-\tilde{y}\right|} \mathrm{d} s_{\tilde{y}}=E.
\end{equation*}

Substituting the above results into \eqref{g(t;y_1)}, a straightforward calculation implies
$$g (t;y_1)=-\frac{1}{E}\frac{\lambda(t-c^{-1}|y_0-s(\tau)|)}
{4\pi|y_0-s(\tau)|\left(1-\frac{v(\tau)\cdot(y_0-s(\tau))}{c|y_0-s(\tau)|}\right)}+\frac{1}{E}O( \varepsilon_1(t)).$$
Then we obtain
\begin{equation*}
\begin{aligned}
u(x,t)&=\frac{Ag(t-c^{-1}|x-y_0|;y_1)}{4\pi|x-y_0|}+O(\varepsilon(x))\\
&=-C\frac{\lambda(\tilde{t}-c^{-1}|y_0-s(\tilde{\tau})|)}
{4\pi|x-y_0||y_0-s(\tilde{\tau})|\left(1-\frac{v(\tilde{\tau})\cdot(y_0-s(\tilde{\tau}))}{c|y_0-s(\tilde{\tau})|}\right)}+O(\varepsilon(x, t)),
\end{aligned}  
\end{equation*}
where $C=\frac{A}{E} $, $\tilde{t}=t-c^{-1}|x-y_0|$, $\tilde{\tau}$ satisfies $\tilde{t}=\tilde{\tau}+c^{-1}|y_0-s(\tilde{\tau})|$, $\varepsilon(x, t)=\max \left \{ \varepsilon (x),  \varepsilon_1(\tilde{t}) \right \} =\max \left \{ \varepsilon (x),  \varepsilon(s(\tilde{t})), \varepsilon(s(\tilde{\tau} )) \right \}  $.

\textup{(2)}~If $|v(t)|= \varepsilon(v(t))c$, we have
\begin{equation*}
    1-\frac{v(\tilde{\tau})\cdot(y_0-s(\tilde{\tau}))}{c|y_0-s(\tilde{\tau})|}=1+O(\varepsilon(v(\tilde{\tau}))).
\end{equation*}
Then a direct computation implies the conclusion.
\qed
\end{proof}

\section{ Direct sampling method for inverse scattering }

The inverse scattering problem in consideration is as follows: reconstruct the scatterer from the time-dependent wave field data 
\begin{equation}\label{scattered_data}
\left \{ u(x,t): x\in \Gamma_m, ~t\in [0,T^{tot}] \right \},
\end{equation}
where $T^{tot}$ is terminal time.

Let $D_z$ represent the sampling region. If $|v(t)|<c$, for a sampling point $z\in D_z$, define
\begin{equation*}
U(x,t;z)=-\frac{\lambda(\tilde{t}-c^{-1}|s(\tilde{\tau})-z|)}
{4\pi|x-z||s(\tilde{\tau})-z|\left(1-\frac{v(\tilde{\tau})\cdot(z-s(\tilde{\tau}))}{c|z-s(\tilde{\tau})|}\right)},
\end{equation*}
where $\tilde{t}=t-c^{-1}|x-z|$, $\tilde{\tau}$ satisfies $\tilde{t}=\tilde{\tau}+c^{-1}|z-s(\tilde{\tau})|$. According to Proposition ~\ref{Proposition: approximate solution}, when the sampling point $z$ is the geometric center of the scatterer $D$, the approximate solution to the corresponding forward scattering problem is $CU(x,t;z)$.

Denote the inner product in $L^2(\Gamma_m\times \mathbb{R})$ by
\begin{equation*}
    \langle u(x,t),v(x,t) \rangle=\int_{\mathbb{R}}\int_{\Gamma_m}u(x,t)\overline{v(x,t)}\mathrm{d}s_x\mathrm{d}t
\end{equation*}
and the norm in $L^2(\Gamma_m\times \mathbb{R})$ by
\begin{equation*}
    \left \| u(x,t) \right \| = \langle u(x,t), u(x,t) \rangle^{\frac{1}{2}}.
\end{equation*}
Considering the approximate solution $CU(x,t;z)$, a basic idea is to construct an indicator function
\begin{equation}\label{I_1}
I_{1}(z)= \frac{\left|\langle u(x,t), U(x,t;z)\rangle \right|}{||u(x,t)|| \cdot ||U(x,t;z)||} , \quad \quad z\in D_z.
\end{equation}
The analysis of the validity of the indicator function \eqref{I_1} is classic and similar analysis can be seen in the cited references such as \cite{CB point-like scatterer}. 

Apart from the indicator function \eqref{I_1}, we want to construct a more concise and effective indicator function to reconstruct the scatterers.

Define a new indicator function
\begin{equation}
    I_2(z)=\left \| \int_{\Gamma _{m} } \sqrt{|x-y_0|} \cdot(u*G_z)(x,t)\mathrm{d}s_x  \right \|^2_{L^{
2}(\mathbb{R})},\quad\quad z\in D_z,
\label{I_2}
\end{equation}
where $y_0$ is the geometric center of the scatterer $D$, $u$ is the solution to the forward scattering problem \eqref{eq:problem1}-\eqref{eq:problem3}, and 
\begin{equation*}
    G_z(x,t)=\frac{\lambda(t+c^{-1}|x-z|)}{4\pi\sqrt{|x-z|}}.
\end{equation*}

For the analysis of the indicator function, we recall the Fourier transform
\begin{equation*}
\mathscr{F}[f](\xi)=\int_{\mathbb{R}}f(t) \mathrm{e}^{-i\xi  t} \mathrm{d}t, \quad \quad \xi \in \mathbb{R}
\end{equation*}
and the inverse Fourier transform
\begin{equation*}
\mathscr{F}^{-1}[f](t)=\frac{1}{2\pi }\int_{\mathbb{R}}f(\xi) \mathrm{e}^{i\xi  t} \mathrm{d}\xi,\quad \quad t \in \mathbb{R}.
\end{equation*}

\begin{Lemma}\label{lemma_integral} 
Denote by $\Gamma_m$ a spherical surface with the center at the origin and a radius of $R$. For any $z\in\mathbb{R}^3$ with $|z|\le R$, we have
\begin{equation*}
\int_{\Gamma_m}\frac{1}{|x-z|} \mathrm{d}s_x= 4\pi R.
\end{equation*}
\end{Lemma}

\begin{proof}
When $z=(0,0,0)\in \mathbb{R}^3$, the conclusion is obvious.

When $0<|z|\leq R $, there exists a rotation matrix $Q$ such that $Qz=(0,0,|z|)$. Since $|x-z|=|Qx-Qz|$ and the rotation preserves the area element $\mathrm{d}s_x=\mathrm{d}s_{Qx}$, we can get
\begin{equation*}
    \int_{\Gamma_m}\frac{1}{|x-z|} \mathrm{d}s_x=\int_{\Gamma_m}\frac{1}{|Qx-Qz|} \mathrm{d}s_{Qx}.
\end{equation*}
Thus, it is reasonable to assume that $z=(0,0,|z|)$. Parameterize  $x\in\Gamma_m$ as
\begin{equation*}
    x=(R\sin\varphi \cos\theta, R\sin\varphi \sin\theta, R\cos\varphi ) , \quad \quad \varphi \in [0,\pi],  \theta\in [0,2\pi].
\end{equation*}
Then we have $\mathrm{d}s_x=R^2 \sin\varphi\mathrm{d}\varphi \mathrm{d}\theta$ and
\begin{equation*}\begin{aligned}
\int_{\Gamma_m}\frac{1}{|x-z|} \mathrm{d}s_x&=\int_{0}^{2\pi} \int_{0}^{\pi} \frac{1}{\sqrt{|z|^2-2|z|R \cos\varphi+R^2 }}\cdot R^2 \sin\varphi \mathrm{d}\varphi \mathrm{d}\theta\\
&=-2\pi R^2 \int_{0}^{\pi} \frac{1}{\sqrt{|z|^2-2|z|R \cos\varphi+R^2 }}  \mathrm{d}\cos\varphi\\
&= \frac{2\pi R }{|z|}\left[\sqrt{(|z|+R)^2}-\sqrt{(|z|-R)^2} \right]\\
&=4\pi R.
\end{aligned}\end{equation*}
\qed
\end{proof}

\begin{Theorem}
Assume that the point-like scatterer $D$ is expressed as $D=\left \{ y_0 +\gamma  \eta~|~\eta \in D_0 \right \}$, where $y_0$ is the geometric center of $D$, $0<\gamma \ll 1$ is a parameter indicating the size of $D$, and $D_0$ is a region with the geometric center $(0,0,0)$ and the diameter $d(D_0):=\max\limits_{x,y\in D_0}|x-y|=1$.
The sampling region is chosen as $D_z$ which satisfies $D_z \supset \overline{D}$. The measurement surface $\Gamma_m$ is chosen as a spherical surface with the center at the origin and a radius of $R$, which contains $D_z$ in it. The signal function is chosen as $\lambda(t)=\delta(t)$. Then the indicator function defined by \eqref{I_2} satisfies
\begin{equation*}
    I_2(z)\leq \frac{\gamma ^4 R^2}{16\pi^2}\int_\mathbb{R}  \left|\int_{\partial D_0} \hat{g}(\xi;y_0+\gamma \eta ) \mathrm{d}s_{\eta} \right|^{2}    \mathrm{d}\xi \cdot [1+O(\gamma)],\quad\quad z\in D_z,
\end{equation*}
where $\hat{g}=\mathscr{F}[g]$, and $g$ is the solution to the boundary integral equation \eqref{eq:boundary integral equation}. Moreover, in the above equation, the equality sign holds when $z = y_0$.
\end{Theorem}

\begin{proof}
Combining the convolution theorem with the Plancherel theorem, we have
\begin{equation*}
    \begin{aligned}
I_2(z)&=\left \|\mathscr{F}\left[ \int_{\Gamma _{m}} \sqrt{|x-y_0|} \cdot(u*G_z)(x,t) \mathrm{d}s_x\right]  \right \|_{L^{
2}(\mathbb{R})}^{2}\\
&=\left \| \int_{\Gamma _{m}} \sqrt{|x-y_0|}\cdot\hat{u}(x,\xi) \cdot\hat{G_z}(x,\xi)]\mathrm{d}s_x  \right \|_{L^{
2}(\mathbb{R})}^{2},\\
\end{aligned}
\end{equation*}
where $\hat{u}=\mathscr{F}[u]$, $\hat{G_z}=\mathscr{F}[G_z]$. When the signal function is chosen as $\lambda (t)=\delta (t) $, we have
\begin{equation*}
\begin{aligned}
I_2(z)&=\left \| \int_{\Gamma _{m} } \sqrt{|x-y_0|}\cdot\hat{u}(x,\xi)\cdot \frac{\mathrm{e}^{i\xi  c^{-1}|x-z|}}{4\pi \sqrt{|x-z|}} \mathrm{d}s_x \right \|_{L^{2}(\mathbb{R})}^{2}\\
&= \int_\mathbb{R} \left| \int_{\Gamma _{m} }  \frac{\sqrt{|x-y_0|}\cdot \hat{u} (x,\xi)\cdot \mathrm{e}^{i\xi  c^{-1}|x-z|}} {4\pi \sqrt{|x-z|}} \mathrm{d}s_x \right|^{2} \mathrm{d}\xi=: \int_\mathbb{R} \left| H(z, \xi) \right|^{2} \mathrm{d}\xi.\\
\end{aligned}
\end{equation*}

Denote $\hat{G}=\mathscr{F}[G]$, $\hat{g}=\mathscr{F}[g]$, $\hat{u^i}=\mathscr{F}[u^i]$. Taking into account the boundary integral equation method \eqref{Kirchhoff}-\eqref{eq:boundary integral equation}, $\hat{u}(x,\xi)$ can be expressed as
\begin{equation}
\hat{u}(x,\xi)=\int_{\partial D}\hat{G}(x,\xi;y)\hat{g}(\xi;y) \mathrm{d}s_y, \quad \quad x\in \mathbb{R}^{3}\setminus \overline{D},  \
\label{usBDE}
\end{equation}
where the density function $\hat{g}(\xi;y)$ satisfies the boundary integral equation
\begin{equation*}
\int_{\partial D }\hat{G}(x,\xi;y) \hat{g}(\xi;y) \mathrm{d}s_y=-\hat{u^{i}}(x,\xi), \quad \quad x\in \partial D.  
\end{equation*}

For $y\in\partial D$, since $y=y_0+\gamma  \eta$, we have $\mathrm{d}s_y=\gamma ^2 \mathrm{d}s_\eta $. Substituting $\hat{G}(x,\xi;y)=\frac{\mathrm{e}^{-i\xi c^{-1}|x-y| } }{4\pi |x-y|}$ into \eqref{usBDE} implies
\begin{equation*}
\begin{aligned}
  \hat{u}(x,\xi)&=\int_{\partial D}\frac{\mathrm{e}^{-i\xi  c^{-1}|x-y|}}{4\pi |x-y| }\hat{g}(\xi;y) \mathrm{d}s_y   \\
  &=\int_{\partial D_0}\frac{ \mathrm{e}^{-i\xi c^{-1}|x-(y_0+\gamma  \eta )|}}{4\pi |x-(y_0+\gamma  \eta )| }\hat{g}(\xi;y_0+\gamma  \eta ) \gamma ^2\mathrm{d}s_{\eta} \\
  &=\frac{\gamma ^2 \mathrm{e}^{-i\xi c^{-1}|x-y_0|}}{4\pi |x-y_0| } \int_{\partial D_0}\hat{g}(\xi;y_0+\gamma  \eta ) \mathrm{d}s_{\eta} \cdot  [1+O(\gamma)].\\
\end{aligned}
\label{intGg}
\end{equation*}
Then we have
\begin{equation*}
H(z, \xi)=\frac{\gamma ^2}{16\pi^2}\int_{\Gamma _{m} } \frac{ \mathrm{e}^{-i\xi c^{-1}(|x-y_0|-|x-z|)}}{\sqrt{|x-y_0|}\cdot\sqrt{|x-z|}}  \mathrm{d}s_x \cdot  \int_{\partial D_0} \hat{g}(\xi;y_0+\gamma  \eta ) \mathrm{d}s_{\eta} \cdot  [1+O(\gamma)].
\label{daiLemma}
\end{equation*}
Then, combining with Lemma \ref{lemma_integral}, it yields
\begin{equation*}
\begin{aligned}
\left|H(z, \xi)\right|^{2}&=\frac{\gamma ^4}{16^2 \pi^4} \left|\int_{\Gamma _{m} } \frac{ \mathrm{e}^{-i\xi c^{-1}(|x-y_0|-|x-z|)}}{\sqrt{|x-y_0|}\cdot\sqrt{|x-z|}}  \mathrm{d}s_x \right|^{2} \cdot  \left|\int_{\partial D_0} \hat{g}(\xi;y_0+\gamma \eta ) \mathrm{d}s_{\eta} \right|^{2} \cdot  [1+O(\gamma )]\\
&\leq \frac{\gamma ^4}{16^2 \pi^4}  \int_{\Gamma _{m} } \frac{1}{|x-y_0|}  \mathrm{d}s_x  \cdot \int_{\Gamma _{m} } \frac{1}{|x-z|}  \mathrm{d}s_x  \cdot  \left|\int_{\partial D_0} \hat{g}(\xi;y_0+\gamma  \eta ) \mathrm{d}s_{\eta} \right|^{2} \cdot  [1+O(\gamma )]\\
&=\frac{\gamma ^4 R^2}{16\pi^2} \left|\int_{\partial D_0} \hat{g}(\xi;y_0+\gamma \eta ) \mathrm{d}s_{\eta} \right|^{2} \cdot  [1+O(\gamma)].\\
\end{aligned}
\end{equation*}
In particular, the above equality sign holds when $z=y_0$. 
\qed
\end{proof}

In the indicator function $I_2(z)$ defined by \eqref{I_2}, the geometric center $y_0$ of the scatterer $D$ is used. But in practice, $y_0$ remains unknown during computation. Nevertheless, the term $\sqrt{|x-y_0|}$ is independent of the sampling point $z$, so a modified indicator function
\begin{equation*}
    \tilde{I}_2 (z)=\left \| \int_{\Gamma _{m} } (u*G_z)(x,t)\mathrm{d}s_x  \right \|_{L^{
2}(\mathbb{R})}^{2},\quad\quad z\in D_z
\end{equation*}
is used in the numerical computation. 

The direct sampling method is as follows: Choose the sampling region $D_z$ and the sampling points $z\in D_z$. Compute the value of $I_1(z)$ or $\tilde{I}_2 (z)$. Then the reconstruction of the location of the scatterer is given by the region where the values of the indicator function are relatively large. The effectiveness of the proposed method is verified by the numerical experiments in Section 4.

\section{ Numerical Experiments}
In this section, comprehensive numerical examples, in both two-dimensional and three-dimensional spaces, are presented to illustrate the feasibility of the direct sampling methods. Both point-like scatterers and extended scatterers can be effectively reconstructed. In all the experiments, only a single moving emitter is used for the reconstruction. 

The boundary integral equation method is used to solve the forward scattering problem to obtain time-dependent scattered field data \eqref{scattered_data}. Random noises are added by the following formula
\begin{equation*}
    u_{\sigma}=(1+\sigma  r)u,
\end{equation*}
where $\sigma  > 0$ is the noise level and $r$ are uniformly distributed random numbers in $\left [ -1,1 \right ] $.

The sampling points are chosen as $z_l$ in the sampling region $D_z$ for $l=1,2,\cdots,N_z$ with $N_z\in \mathbb{N}^*$. The time discretization is
\begin{equation*}
    t_k=k\Delta t, \quad \Delta t=\frac{ T^{tot} }{N_t},\quad \quad k=0,1,2,\cdots,N_t,
\end{equation*}
where $T^{tot}$ is the terminal time, $N_t\in \mathbb{N}^*$. The measurement points are chosen as $x_i\in \Gamma_m$ with the control area $\Delta s_{x_i}$ for $i=1,2,\cdots,N_m$, where $N_m\in \mathbb{N}^*$.
Then we compute the discretized indicator functions
\begin{equation*}
    I_1(z_l)=\frac{\left |\sum_{k=0}^{N_t-1}  \sum_{i=1}^{N_m} u(x_i,t_{k}) \overline{U(x_i,t_{k};z_l)} \Delta s_{x_i}  \Delta t \right  |}
{\left(\sum_{k=0}^{N_t-1}  \sum_{i=1}^{N_m} |u(x_i,t_{k})|^2\Delta s_{x_i} \Delta t\right)^{1/2}
\left(\sum_{k=0}^{N_t-1}  \sum_{i=1}^{N_m} |U(x_i,t_{k};z_l)|^2\Delta s_{x_i} \Delta t\right)^{1/2} }
\end{equation*}
and
\begin{equation*}
    \tilde{I}_2(z_l)=\sum_{k=0}^{N_t-1} \left | \sum_{i=1}^{N_m} \sum_{j=0}^{k} u(x_i,t_{k-j})G_{z_l}(x_i,t_{j})\Delta t\Delta s_{x_i} \right  |^{2} \Delta t
\end{equation*}
to reconstruct the scatterers.

\subsection{Two-dimensional Experiments}

Consider a two-dimensional region of $[-36,36]\times[-36,36]$ as the sampling region. The sampling points are chosen as $36\times 36$ uniform discrete points in the sampling region. The measurement points $x$ are located on a circle centered at the origin with a radius of $R_m=72$. The sound speed is $c=340$. The moving point source follows a circular trajectory with a radius $R_S=60$, starting at the position $(60,0)$ and moving counterclockwise. Let $\omega $ be the angular velocity of the moving point source. Then, the trajectory of the moving point source in $\mathbb{R}^{2}$ is given by
\begin{equation*}
    s(t)=R_S(\cos \omega t, \sin \omega t).
\end{equation*}

Assume that the moving point source rotates for one full circle during the time interval of $[0, T]$, in which $ T=14$. Then the angular speed of the moving emitter is $\omega_0=\frac{2\pi}{T} $. The signal function is selected as a periodic function $\lambda_N(t)$ with a period of $\frac{ T }{N}$, where $N\in \mathbb{N}^*$ is the number of periods. The expression of $\lambda_N(t)$ within $[0,\frac{ T }{N}]$ is 
\begin{equation*}
    \lambda_N(t)=\sin(10N t)\mathrm{e}^{-15N^2(t-\frac{ T }{3N} )^2}, \quad \quad  t\in [0,\frac{ T }{N}].
\end{equation*}
The signal functions $\lambda_N(t)$ with $N=1$, $3$ and $10$ are shown in Figure \ref{incident pulse}. The time discretization is accomplished with $N_t=256N$.

\begin{figure}[!htbp]
		\centering
		\begin{tabular}{ccc}
			\includegraphics[width=0.3\textwidth]{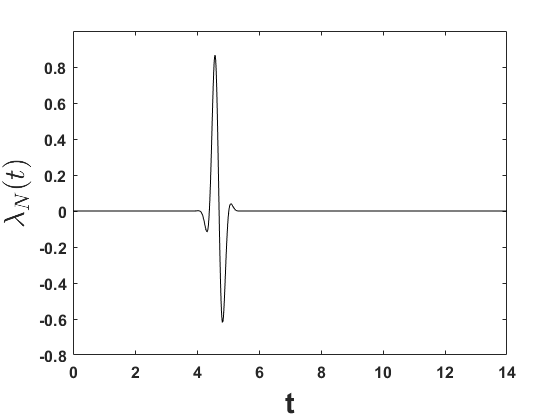}
             & \includegraphics[width=0.3\textwidth]{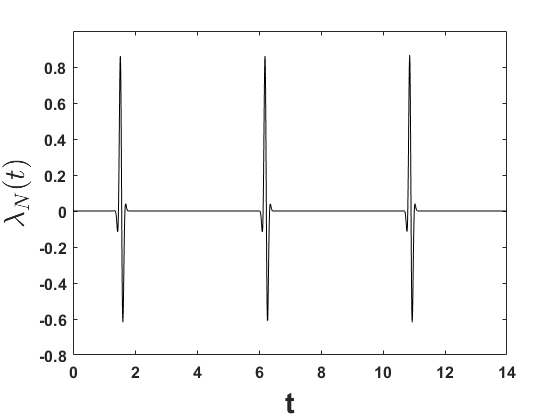}
            & \includegraphics[width=0.3\textwidth]{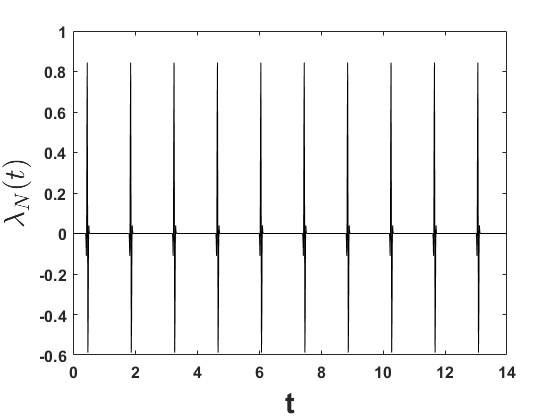}\\
            (a)~$N=1$ & (b)~ $N=3$ & (c)~$N=10$ \\
		\end{tabular}
\caption{The signal function $ \lambda_N(t)$.}
\label{incident pulse}
\end{figure}

\subsubsection*{Example 1: Comparison of $I_1(z)$ and $\tilde{I}_2 (z)$ with different signal functions}

The parameterized boundary of a circular scatterer centered at $(a,b)$ is chosen as
\begin{equation}\label{ball_shaped}
(a+R_b\cos\theta, b+R_b\sin\theta ),   
\end{equation}
where $R_b$ is the radius of the circular scatterer, $\theta  \in [0,2\pi]$. For all the point-like circular scatterers in our experiments, we choose $R_b=0.01$.

In this example, both the reconstructions of multiple point-like scatterers and an extended scatterer are considered. The indicator functions are chosen as $I_1(z)$ and $\tilde{I}_2(z)$. The numbers of periods are chosen as $N=1$, $3$ and $10$. The reconstructions of three point-like scatterers which are centered at $(-24,-24)$, $(0,20)$ and $(15,-10)$ are shown in Figure \ref{point-like scatterer with N=1,3,10}. 
Figure \ref{extended scatterer with N=1,3,10} shows the reconstructions of a circular scatterer which is centered at $(0,0)$ with a radius of $R_b=10$.

In all the figures in this section, the thick curve represents the moving path of the moving emitter, and the black asterisks indicate the receivers. The green asterisks indicate the actual positions of the point-like scatterers, and the dotted lines represent the boundaries of the extended scatterers.

As shown in  Figure \ref{point-like scatterer with N=1,3,10} and
Figure \ref{extended scatterer with N=1,3,10}, the reconstruction effect improves as $N$ increases. When we choose $N=10$, both $I_1(z)$ and $\tilde{I}_2(z)$ provide effective reconstructions, and the indicator function $\tilde{I}_2(z)$ shows better performances.

\begin{figure}[!htbp]
		\centering
		\begin{tabular}{ccc}
            \includegraphics[width=0.3\textwidth]{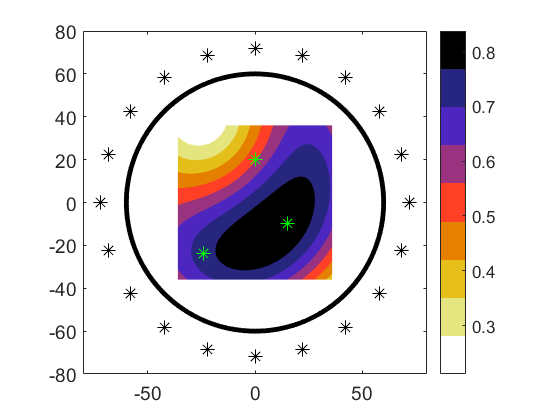}
            & \includegraphics[width=0.3\textwidth]{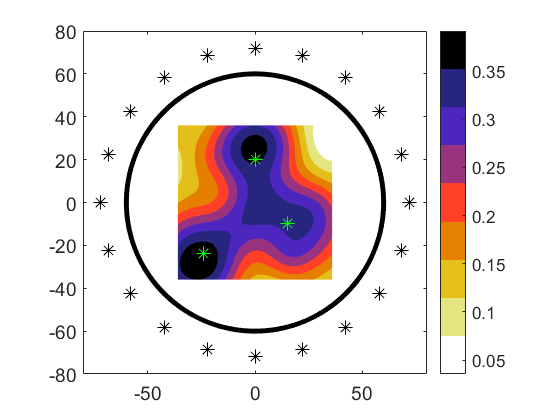}
            & \includegraphics[width=0.3\textwidth]{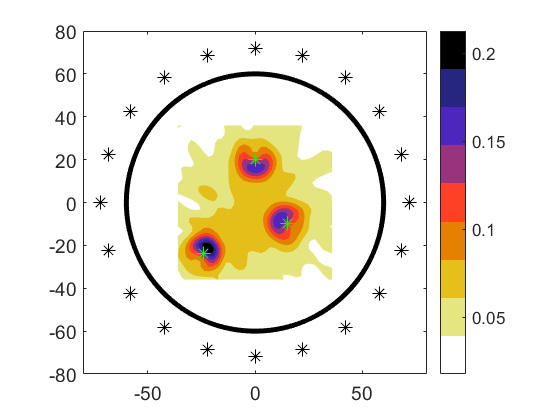}\\
            (a)~$I_1$, $N=1$ & (b)~$I_1$, $N=3$ & (c)~$I_1$, $N=10$ \\
            \includegraphics[width=0.3\textwidth]{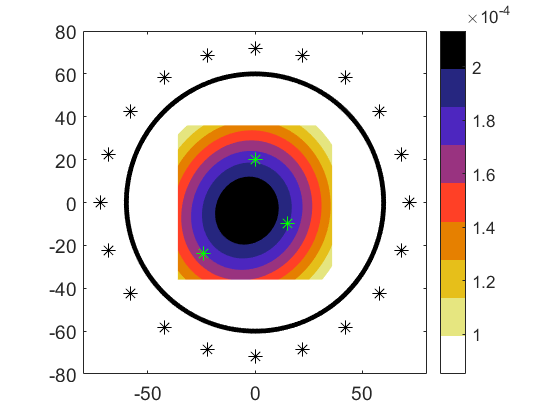}
            & \includegraphics[width=0.3\textwidth]{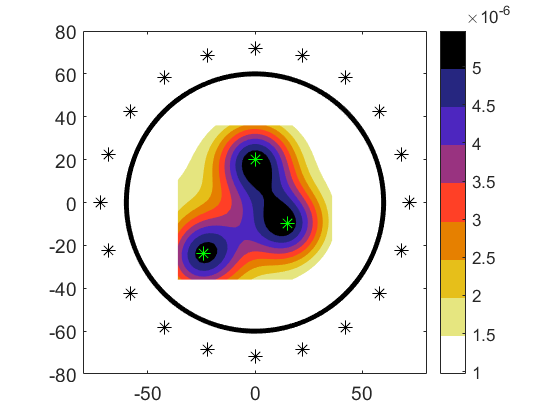}
            & \includegraphics[width=0.3\textwidth]{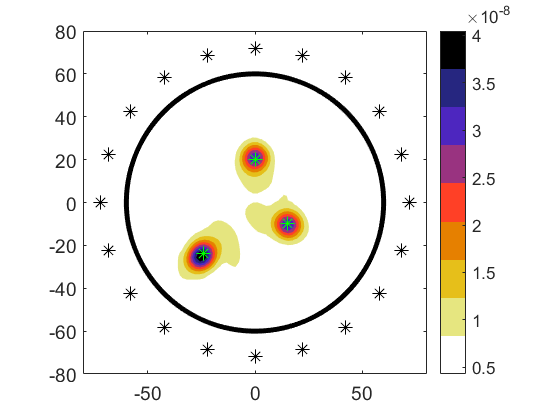}\\
           (d)~$\tilde{I}_2 (z)$, $N=1$ & (e)~$\tilde{I}_2 (z)$, $N=3$ & (f)~$\tilde{I}_2 (z)$, $N=10$ \\         
		\end{tabular}
\caption{The inversion performance of the indicator functions $I_1(z)$ and $\tilde{I}_2(z)$ for multiple point-like scatterers with the signal function $\lambda_N (t) $ under different numbers of periods. The noise level is $\sigma=5\%$.}
\label{point-like scatterer with N=1,3,10}
\end{figure}

\begin{figure}[!htbp]
		\centering
		\begin{tabular}{ccc}
            \includegraphics[width=0.3\textwidth]{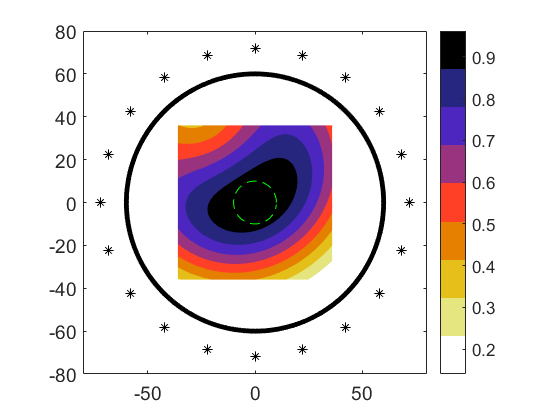}
            & \includegraphics[width=0.3\textwidth]{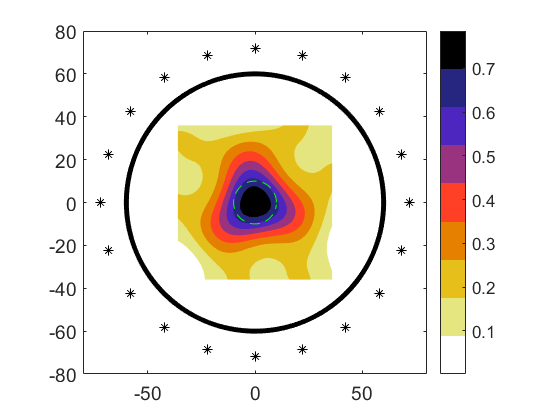}
            & \includegraphics[width=0.3\textwidth]{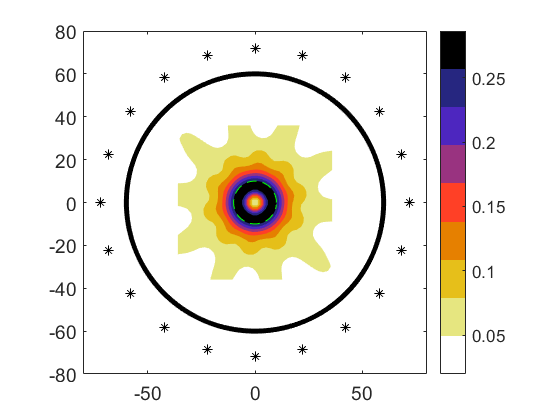}\\
            (a)~$I_1$, $N=1$ & (b)~$I_1$, $N=3$ & (c)~$I_1$, $N=10$ \\
           \includegraphics[width=0.3\textwidth]{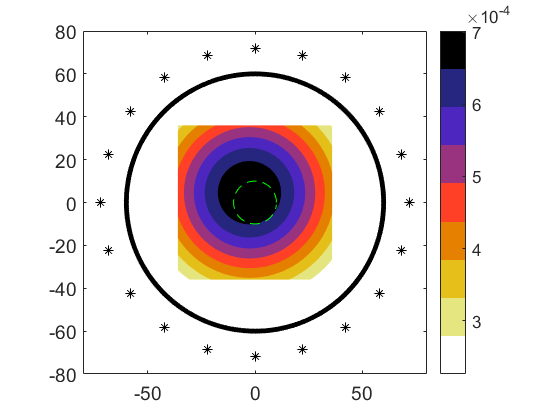}
            & \includegraphics[width=0.3\textwidth]{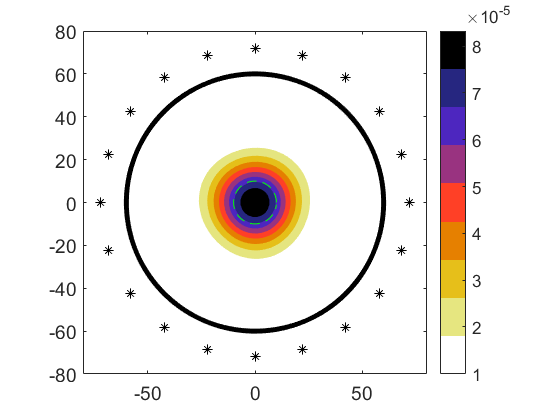}
            & \includegraphics[width=0.3\textwidth]{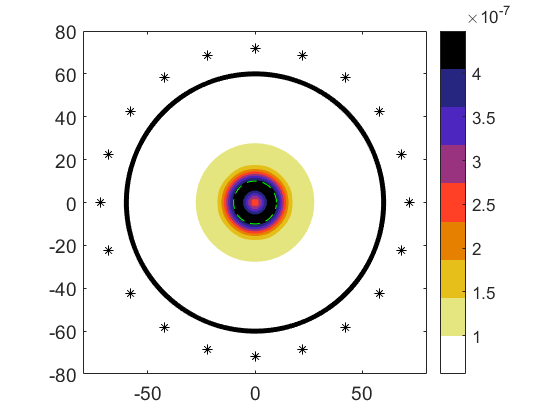}\\   
         (d)~$\tilde{I}_2 (z)$, $N=1$ & (e)~$\tilde{I}_2 (z)$, $N=3$ & (f)~$\tilde{I}_2 (z)$, $N=10$ \\          
		\end{tabular}
\caption{The inversion performance of the indicator functions $I_1(z)$ and $\tilde{I}_2(z)$ for an extended scatterer with the signal function $\lambda_N (t) $ under different numbers of periods. The noise level is $\sigma=5\%$.}
\label{extended scatterer with N=1,3,10}
\end{figure}

\subsubsection*{Example 2: Comparison of $I_1(z)$ and $\tilde{I}_2 (z)$ with different moving speeds of the emitter}

The moving speed is an important factor for a moving point source. This example considers the influence of the moving speed on our indicator functions. The reconstruction of five point-like circular scatterers centered at $(-24,-24)$, $(-24,15)$, $(0,0)$, $(10,-20)$ and $(24,20)$ are considered in this example. Choose $N=10$ as the number of periods of the signal function. The angular speeds of the moving emitter are chosen as $\omega_1=3\omega_0$, $\omega_2=7\omega_0$ and $\omega_3=9\omega_0$. Reconstructions with the indicator functions $I_1(z)$ and $\tilde{I}_2 (z)$ at different moving speeds are shown in Figure \ref{compareI_speed}.

As shown in Figure \ref{compareI_speed}, the indicator function $\tilde{I}_2 (z)$ performs significantly better than $I_1 (z)$. The effect of $I_1 (z)$ is affected by the moving speed, while $\tilde{I}_2 (z)$ remains almost entirely unaffected and can accurately reconstruct the positions of the point-like scatterers at different speeds, even when the moving speed of the emitter approaches the sound speed.

\begin{figure}[!htbp]
		\centering
		\begin{tabular}{ccc}
        \includegraphics[width=0.3\textwidth]{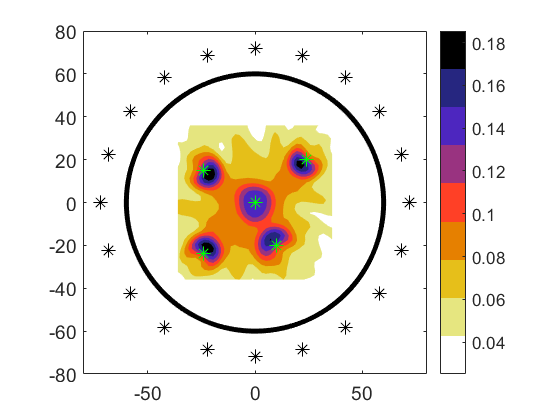}
			& \includegraphics[width=0.3\textwidth]{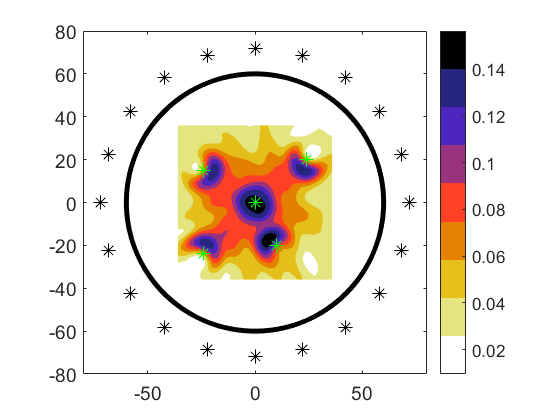}
            & \includegraphics[width=0.3\textwidth]{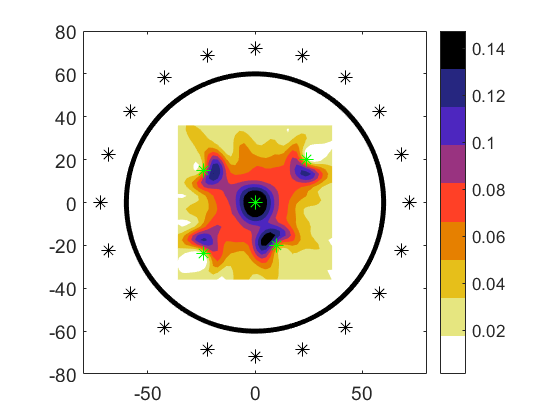} \\
            (a)~$I_1(z), |v_1|=\omega_1 R_S\approx 81$ & 
            (b)~$I_1(z), |v_2|=\omega_2 R_S\approx 188$ & 
            (c)~$I_1(z),|v_3|=\omega_3 R_S\approx 242$ \\
			\includegraphics[width=0.3\textwidth]{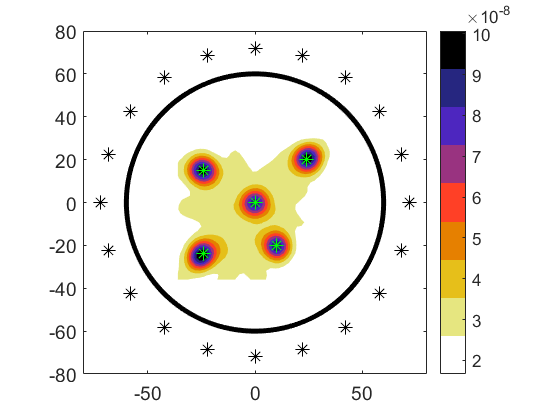}
			& \includegraphics[width=0.3\textwidth]{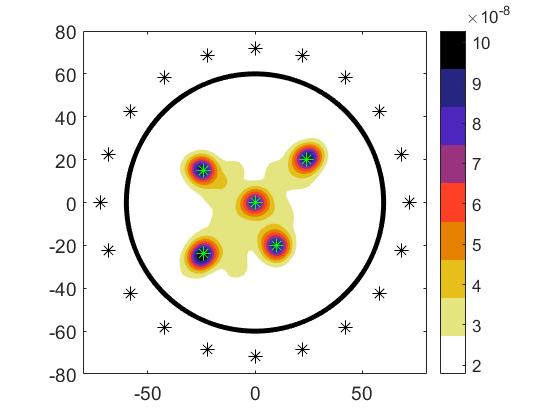}
            & \includegraphics[width=0.3\textwidth]{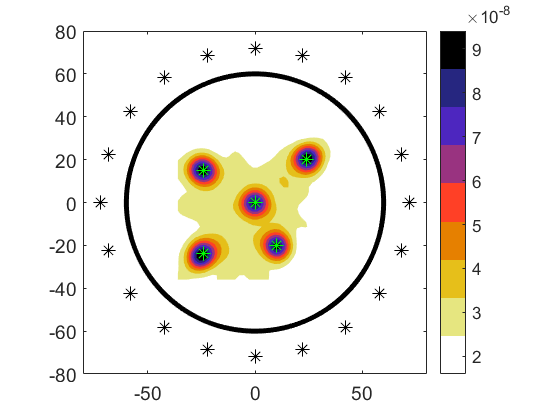} \\
            (d)~$\tilde{I}_2 (z), |v_1|=\omega_1 R_S\approx 81$ & 
            (e)~$\tilde{I}_2 (z), |v_2|=\omega_2 R_S\approx 188$ & 
            (f)~$\tilde{I}_2 (z),|v_3|=\omega_3 R_S\approx 242$ \\
		\end{tabular}
\caption{The inversion performance of the indicator functions $I_1(z)$ and $\tilde{I}_2(z)$  with different moving speeds of the emitter. The noise level is $\sigma=5\%$.}
\label{compareI_speed}
\end{figure}

According to the analysis of Example $1$ and Example $2$, in the following two-dimensional experiments, the indicator function is chosen as $\tilde{I}_2(z)$, the number of periods is chosen as $N=10$, and the angular speed of the moving emitter is chosen as $\omega_0$.

\subsubsection*{Example 3: Reconstruction of multiple point-like scatterers with different noise levels}
In this example, we consider the reconstruction using the indicator function $\tilde{I}_2(z)$ for three cases of multiple point-like scatterers: three point-like scatterers which are centered at $(-24,-24)$, $(0,20)$ and $(15,-10)$, five point-like scatterers which are centered at $(-24,-24)$, $(-24,15)$, $(0,0)$, $(10,-20)$ and $(24,20)$, five point-like scatterers which are centered at $(-28,-28)$, $(-20,-20)$, $(0,0)$, $(8,12)$ and $(20,0)$. The reconstructions with different noise levels can be seen in Figure 
\ref{point-like}.

As shown in Figure \ref{point-like}, the indicator function $\tilde{I}_2(z)$ shows moderately inversion capability for multiple point-like scatterers even if the scatterers are close to each other. Also, the experiment shows that the algorithm is robust against noise.

\begin{figure}[!htbp]
		\centering
		\begin{tabular}{ccc}
			\includegraphics[width=0.3\textwidth]{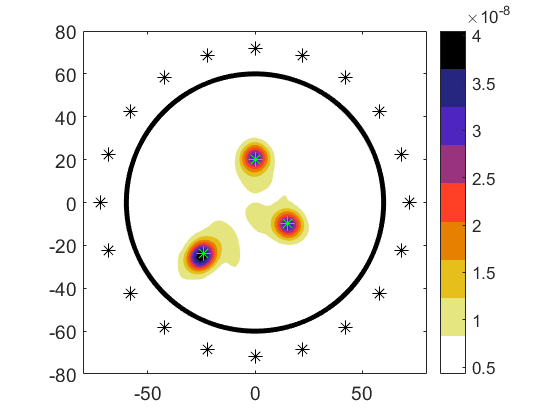}
			& \includegraphics[width=0.3\textwidth]{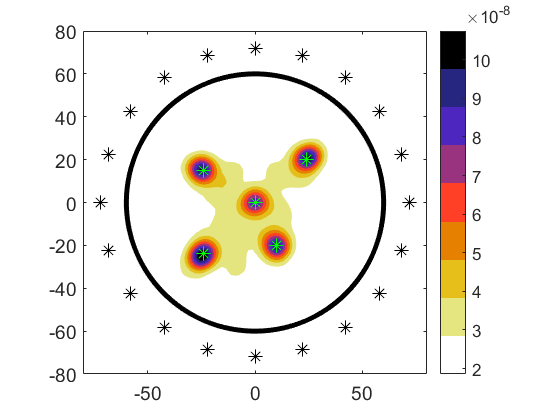}
            & \includegraphics[width=0.3\textwidth]{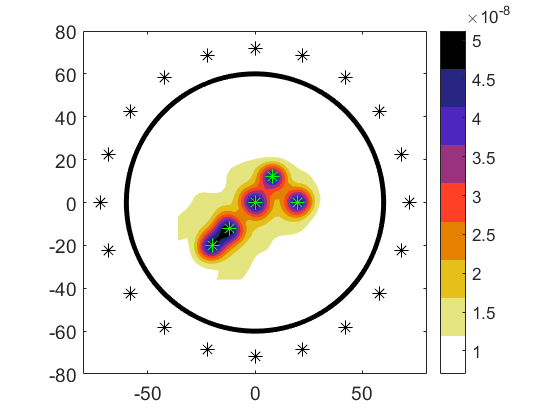} \\
            (a)~$\sigma =5\%$ & (b)~$\sigma =5\%$ & (c)~$\sigma =5\%$ \\
            \includegraphics[width=0.3\textwidth]{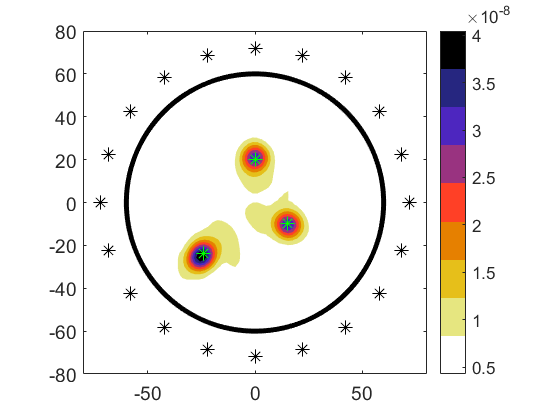}
			& \includegraphics[width=0.3\textwidth]{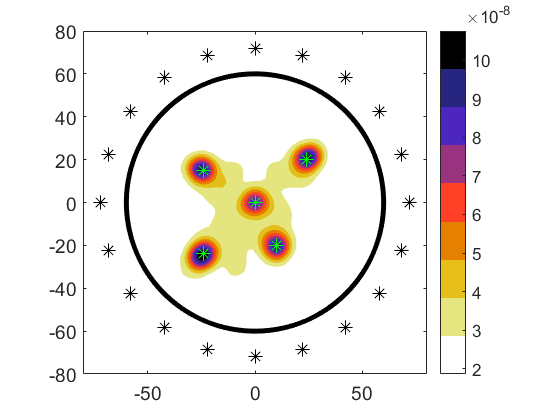}
            & \includegraphics[width=0.3\textwidth]{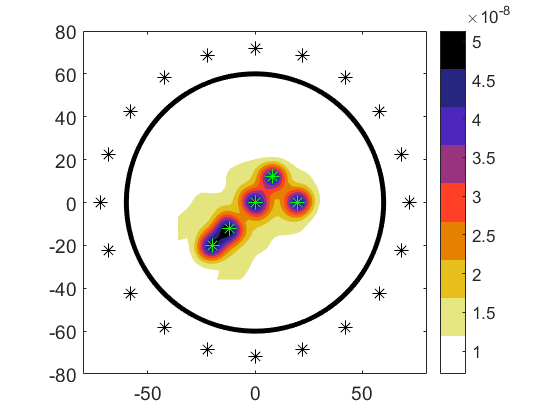} \\
            (d)~$\sigma =20\%$ & (e)~$\sigma  =20\%$ & (f)~$\sigma =20\%$ \\
		\end{tabular}
\caption{The reconstruction of multiple point-like scatterers using the indicator function $\tilde{I}_2 (z)$ with different noise levels. }
\label{point-like}
\end{figure}

\subsubsection*{Example 4: Reconstruction of a single extended scatterer with different noise levels}
The parameterized boundaries of an acorn-shaped scatterer and a square scatterer centered at $(a,b)$ are respectively 
\begin{equation}\label{acorn_shaped}
  (a+R_a\cos\theta\sqrt{\frac{17}{4}+2\cos(3\theta)},
b+R_a\sin\theta\sqrt{\frac{17}{4}+2\sin(3\theta)}~)
\end{equation}
and 
\begin{equation}\label{square_shaped}
 (a+R_q(\sin^{3}\theta+\sin\theta+\cos^{3}\theta+\cos\theta),
b+R_q(\sin^{3}\theta+\sin\theta-\cos^{3}\theta-\cos\theta)),
\end{equation}
where $R_a,R_q\in \mathbb{R}$, $\theta \in [0,2\pi]$.

In this example, we consider the reconstruction of a single extended scatterer in three cases: a circular scatterer given by the parameterized boundary \eqref{ball_shaped} with $(a,b)=(0,0)$ and $R_b=10$, an acorn-shaped scatterer given by \eqref{acorn_shaped} with $(a,b)=(0,0)$ and $R_a=6$ and a square scatterer given by \eqref{square_shaped} with $(a,b)=(-8,-8)$ and $R_q=3\sqrt{2}$. Reconstructions with different noise levels can be seen in Figure \ref{extended}.

As shown in Figure \ref{extended}, the indicator function $\tilde{I}_2(z)$ also achieves effective reconstruction for a single extended scatterer, whether its center is located at the center of the sampling region or not. Additionally, the algorithm is robust to noise.

\begin{figure}[!htbp]
		\centering
		\begin{tabular}{ccc}
        \includegraphics[width=0.3\textwidth]{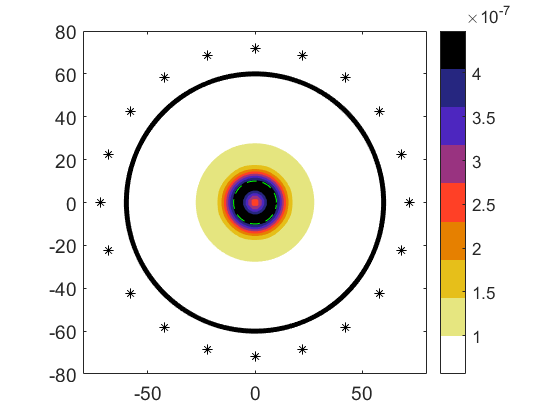}
			& \includegraphics[width=0.3\textwidth]{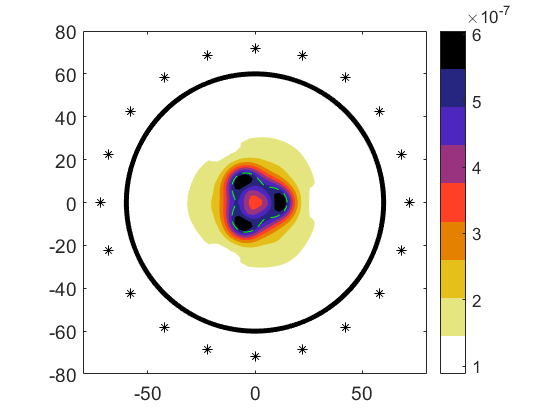}
            & \includegraphics[width=0.3\textwidth]{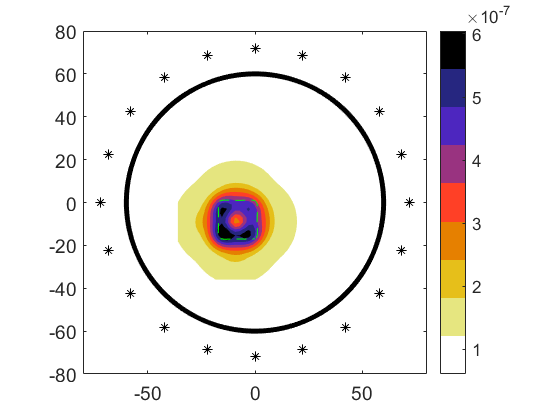} \\
            (a)~ $\sigma  =5\%$ & (b)~ $\sigma  =5\%$ & (c)~ $\sigma  =5\%$ \\
            \includegraphics[width=0.3\textwidth]{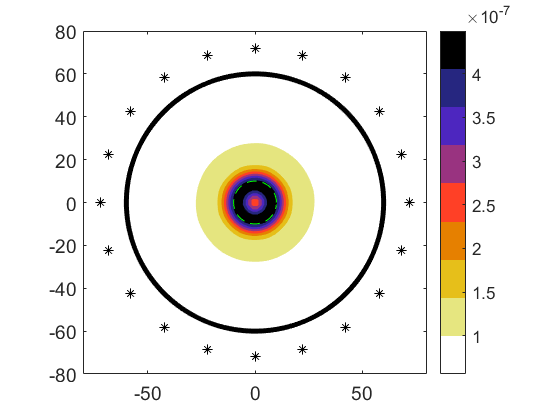}
			& \includegraphics[width=0.3\textwidth]{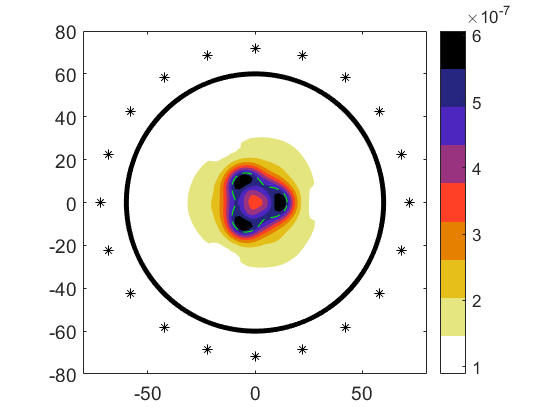}
            & \includegraphics[width=0.3\textwidth]{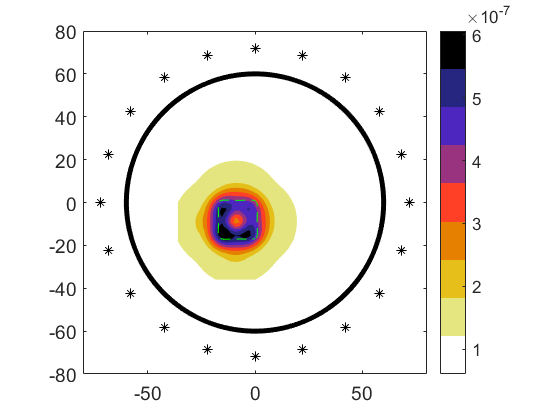} \\
            (d)~$\sigma  =20\%$ & (e)~$\sigma  =20\%$ & (f)~$\sigma  =20\%$ \\
		\end{tabular}
\caption{The reconstruction of a single extended scatterer using the indicator function $\tilde{I}_2(z)$ with different noise levels.}
\label{extended}
\end{figure}

\subsubsection*{Example 5: Reconstruction of a single extended scatterer with limited aperture data}

In this example, we consider the reconstruction of a single extended scatterer with limited aperture data-sets. The scatterers are chosen the same as that in Example $4$. The reconstructions with limited aperture data-sets can be seen in Figure \ref{half}.

As shown in Figure \ref{half}, if the number of receivers is halved, the inversion effect deteriorates. If the moving source only traverses a semicircle under the same moving speed, it results in effective reconstruction only on the emitter-equipped side.

\begin{figure}[!htbp]
		\centering
		\begin{tabular}{ccc}
			 \includegraphics[width=0.3\textwidth]{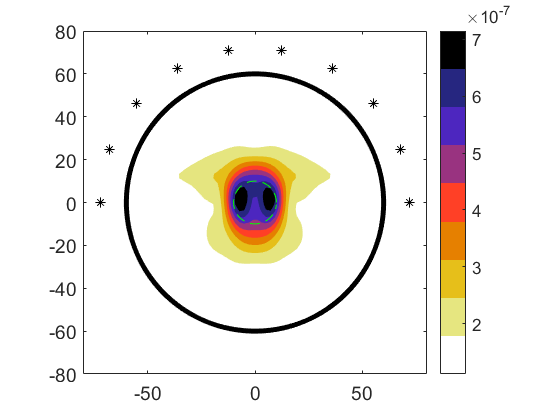}
			& \includegraphics[width=0.3\textwidth]{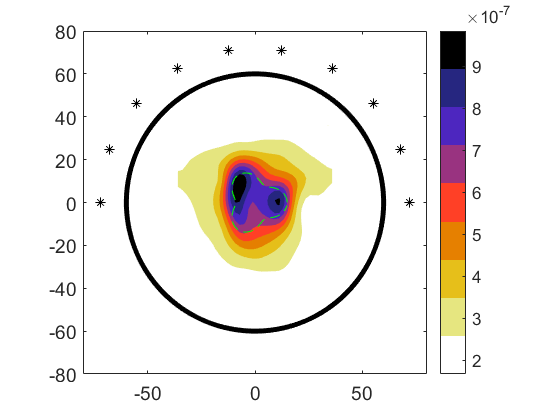}
            & \includegraphics[width=0.3\textwidth]{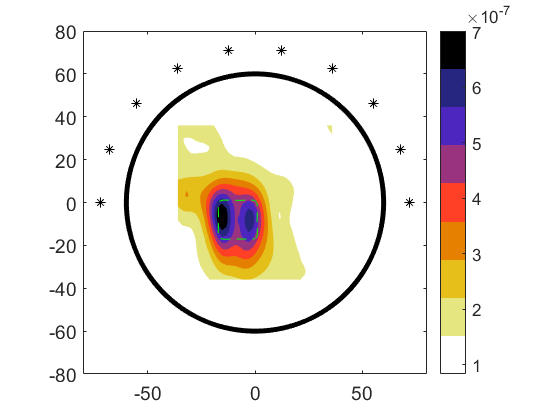} \\
            (a)& (b)& (c)\\
            \includegraphics[width=0.3\textwidth]{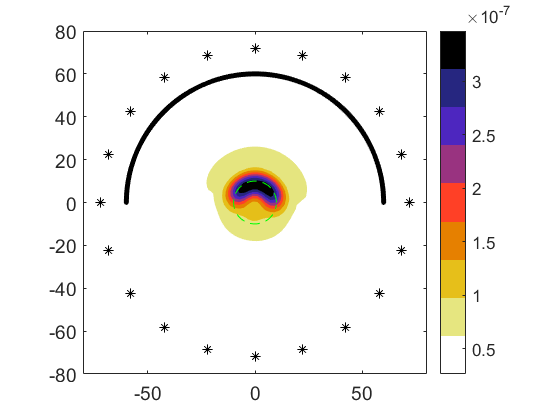}
			& \includegraphics[width=0.3\textwidth]{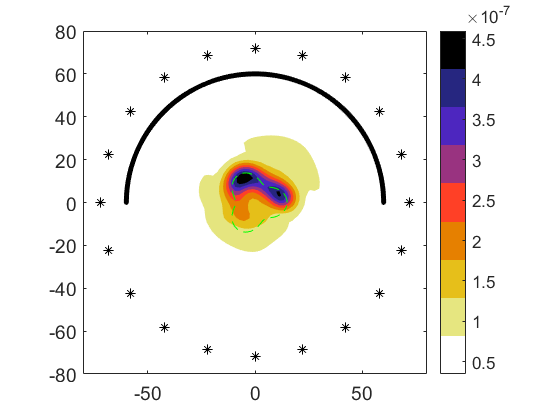}
            & \includegraphics[width=0.3\textwidth]{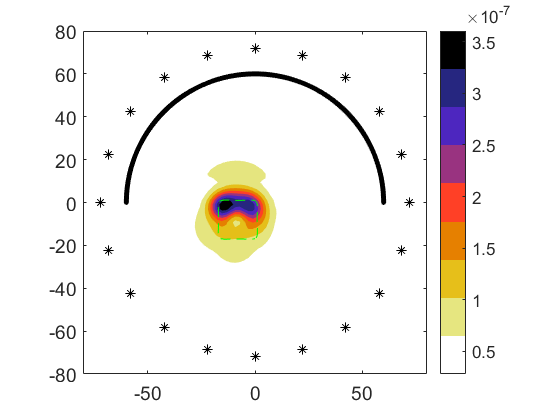} \\
            (d)& (e)& (f)\\
		\end{tabular}
\caption{The reconstruction of a single extended scatterer using the indicator function $\tilde{I}_2(z)$ with limited aperture data. The noise level is $\sigma  =5\%$.}
\label{half}
\end{figure}

\subsubsection*{Example 6: Reconstruction of two disconnected scatterers}
The boundary of a kite-shaped scatterer centered at $(a,b)$ is parameterized as 
\begin{equation}\label{kite-shaped}
(a+R_k(\cos\theta + 0.65 \cos 2\theta -0.65), b+R_k(1.5 \sin\theta )), 
\end{equation}
where $R_k\in\mathbb{R}$, $\theta  \in [0,2\pi]$. 

In this example, we consider the simultaneous reconstruction of two disconnected scatterers with a single moving emitter in three cases: an acorn-shaped scatterer given by \eqref{acorn_shaped} with $(a,b)=(-12,-12)$, $R_a=2.4$ together with a kite-shaped scatterer given by \eqref{kite-shaped} with  $(a,b)=(15,15)$, $R_k=6$, a circular scatterer given by \eqref{ball_shaped} with $(a,b)=(-15,-15)$, $R_b=6$ together with an square scatterer given by \eqref{square_shaped} with $(a,b)=(10,10)$, $R_q=3\sqrt{2}$, an acorn-shaped scatterer given by \eqref{acorn_shaped} with $(a,b)=(-10,-10)$, $R_a=3.6$ together with a circular scatterer given by \eqref{ball_shaped} with $(a,b)=(10,10)$, $R_b=2$. The reconstruction can be seen in Figure \ref{mult-extended}. 

As shown in Figure \ref{mult-extended}, the indicator function $\tilde{I}_2(z)$ achieves satisfactory reconstruction for two disconnected scatterers, even when there exists a significant size disparity between the two scatterers.

\begin{figure}[!htbp]
		\centering
		\begin{tabular}{ccc}
			\includegraphics[width=0.3\textwidth]{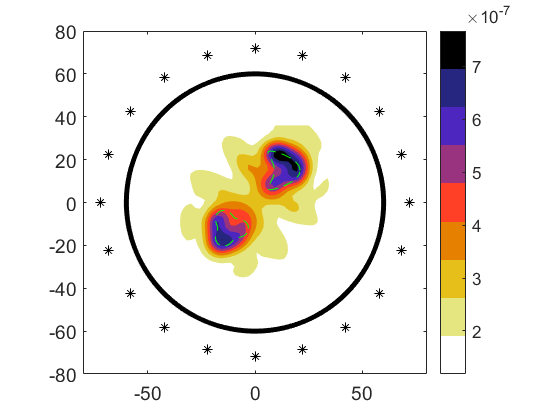}
			& \includegraphics[width=0.3\textwidth]{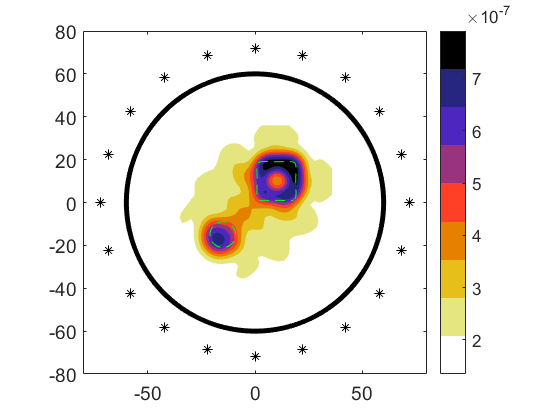}
            & \includegraphics[width=0.3\textwidth]{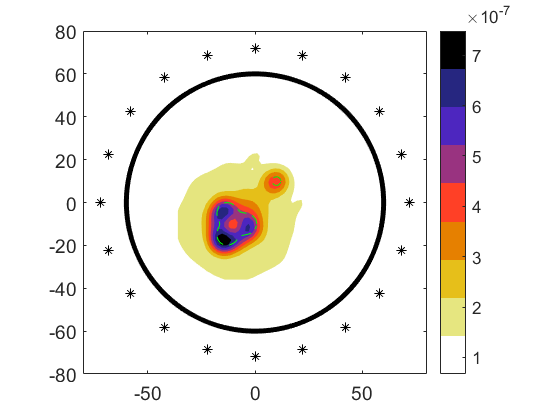}\\
            (a) & (b) & (c) \\
		\end{tabular}
\caption{The reconstruction of two disconnected scatterers using the indicator function $\tilde{I}_2(z)$. The noise level is $\sigma  =5\%$.}
\label{mult-extended}
\end{figure}

\subsection{Three-dimensional Experiments}

Three-dimensional wave data are generated using the k-Wave toolbox \cite{k-wave}. Set the moving trajectory of a single moving emitter in three-dimensional space as a spherical spiral line with radius $R_S=60$, and the parametric representation is
\begin{equation*}
    s(t) = (R_S \sin\alpha(t) \cdot \cos\beta(t), R_S \sin\alpha(t) \cdot \sin \beta(t), R_S \cos\alpha(t)),
\end{equation*}
where $\alpha(t) = \arccos(1-2p(t))$, $\beta(t) = 2 n \pi  p(t)$, $n=5$ denotes the number of spiral turns, $p(t)=\frac{t}{ T^{tot} }$, $t\in [0, T^{tot} ]$. The moving path is presented in Figure \ref{path in 3D} $(a)$. 

The moving point source emits the same period wave $\lambda_N(t)$ in the time interval $[0, T^{tot}]$ with $T^{tot}=3T$, $N=10$ and $T=14$. The positions of $50$ sensors are generated by the makeCartSphere function in the k-Wave toolbox, and they are distributed on a spherical surface with a radius of $R_m=72$, centered at $(0,0,0)$. The sketch of the three-dimensional problem with a cubical scatterer is shown in Figure \ref{path in 3D} $(b)$. The sampling points are chosen as $21\times 21\times 21$ uniform discrete points in the sampling region $[-40,40]\times [-40,40]\times [-40,40]$.

\begin{figure}[!htbp]
		\centering
		\begin{tabular}{ccc}
            \includegraphics[width=0.35\textwidth]{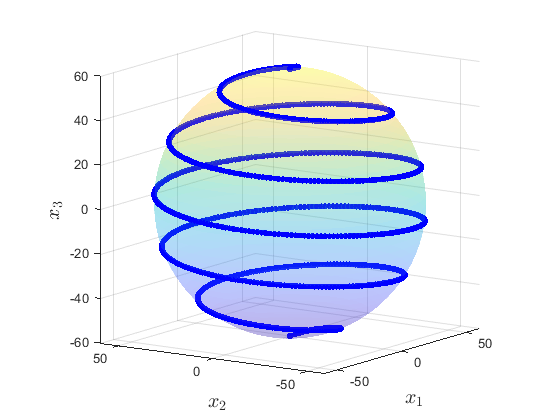}
            & \includegraphics[width=0.4\textwidth]{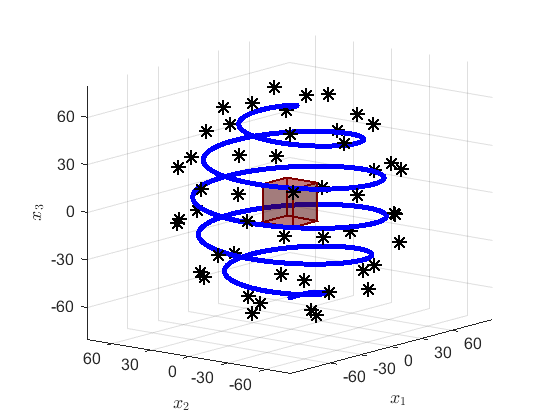}\\
            (a) & (b) 
		\end{tabular}
\caption{(a) The path of a single moving point source in the three-dimensional space. (b) The sketch of the three-dimensional problem with a cubical scatterer. }
\label{path in 3D}
\end{figure}

\subsubsection*{Example 7: Reconstruction of point-like scatterers in the three-dimensional space }

In this example, we consider the reconstruction of a single point-like scatterer located at $(8,-16,4)$ and the reconstruction of two point-like scatterers located at $(-20,-16,-12)$ and $(12,16,20)$ respectively. Figure \ref{1point} shows the inversion results and planar cross-sections of the example. 

As shown in Figure \ref{1point}, the indicator function $\tilde{I}_2(z)$ provides effective reconstructions of point-like scatterers in the three-dimensional space.

\begin{figure}[!]
		\centering
		\begin{tabular}{ccc}
            \includegraphics[width=0.3\textwidth]{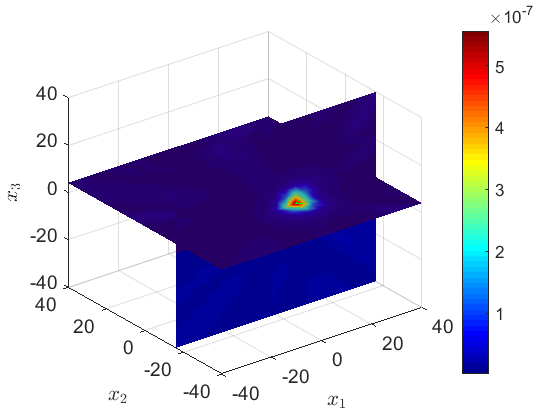}
			& \includegraphics[width=0.3\textwidth]{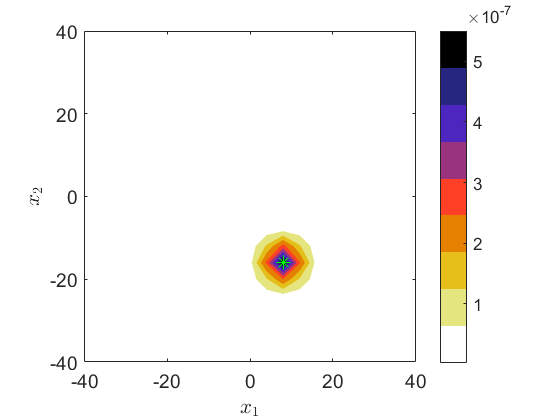}
            & \includegraphics[width=0.3\textwidth]{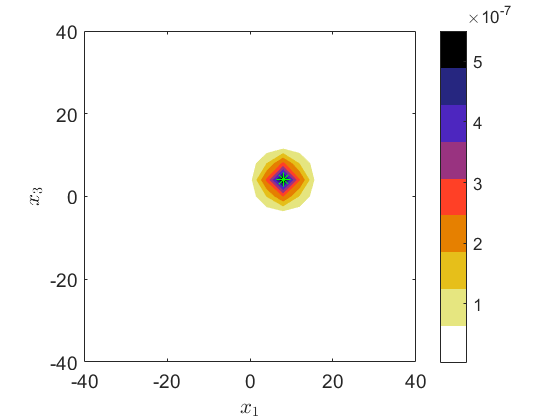} \\
            (a)& (b)& (c)\\
             \includegraphics[width=0.3\textwidth]{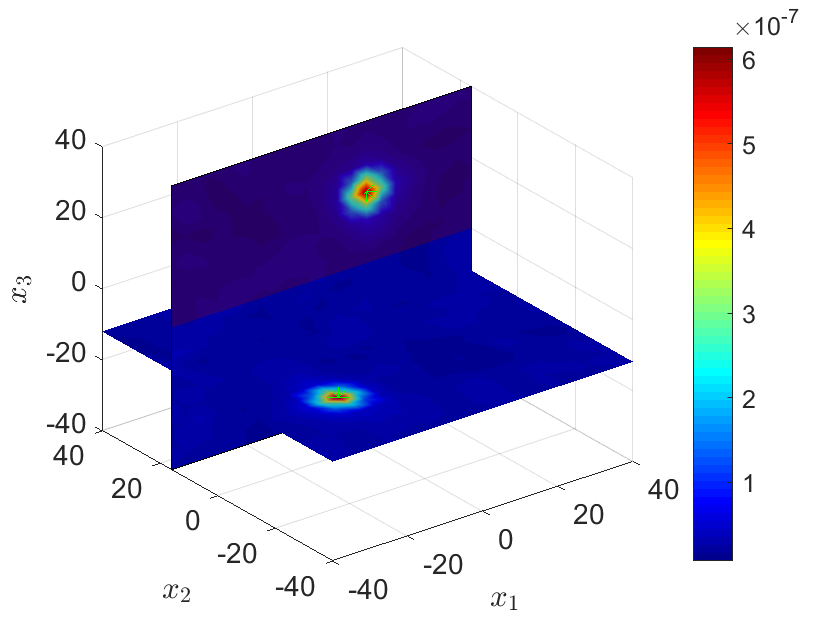}
			& \includegraphics[width=0.3\textwidth]{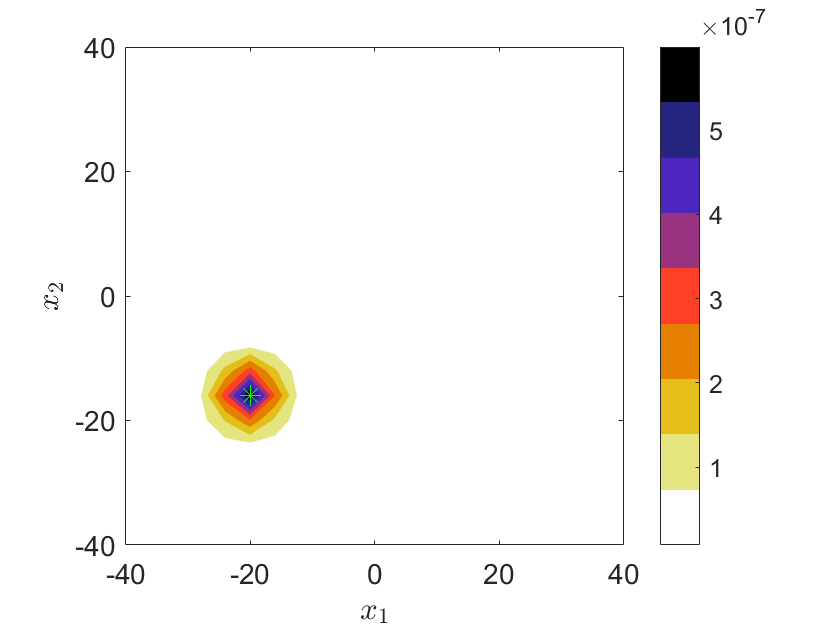}
            & \includegraphics[width=0.3\textwidth]{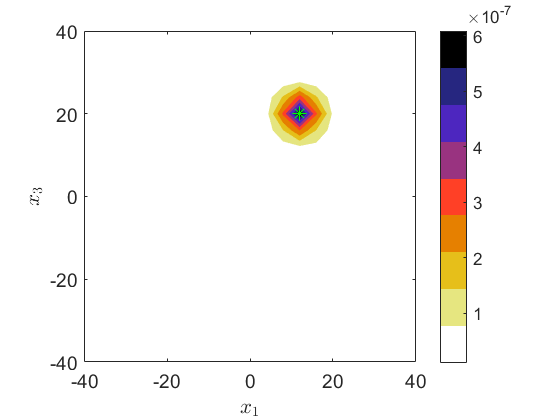} \\
            (d)& (e)  & (f)\\
		\end{tabular}
\caption{The reconstruction of point-like scatterers in the three-dimensional space using the indicator function $\tilde{I}_2(z)$. The noise level is $\sigma  =5\%$. (a,d) The reconstruction in the three-dimensional space. (b,e) A cross-section through the center of a scatterer parallel to the $x_1-x_2$ plane. (c,f) A cross-section through the center of a scatterer parallel to the $x_1-x_3$ plane. }
\label{1point}
\end{figure}

\subsubsection*{Example 8: Reconstruction of extended scatterers in the three-dimensional space }

In this example, we consider the reconstruction of extended scatterers in two cases: a single cube
$[-11,11]\times[-11,11]\times [-11,11]$, two disconnected cubes $[12,24]\times[12,24]\times [12,24]$ and $[-24,-12]\times[-24,-12]\times [-24,-12]$.
Figure \ref{1cube} and Figure \ref{2cubes} demonstrate the effective reconstructions of cubes in the three-dimensional space.

\begin{figure}[!htbp]
		\centering
		\begin{tabular}{ccc}
\includegraphics[width=0.3\textwidth]{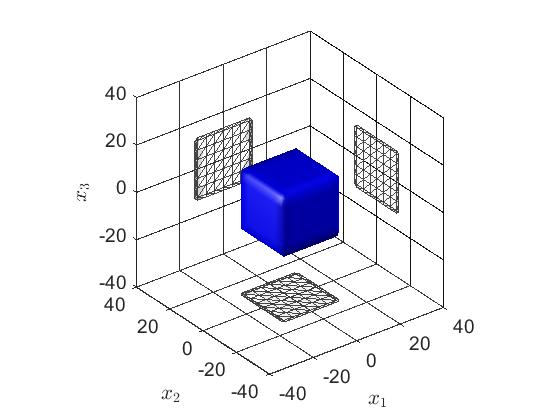}
			& \includegraphics[width=0.3\textwidth]{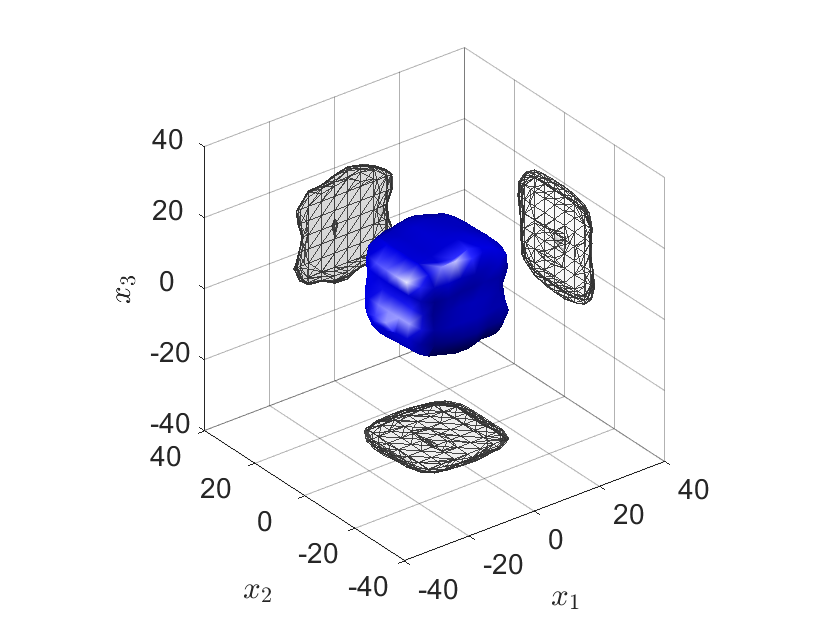}
            & \includegraphics[width=0.3\textwidth]{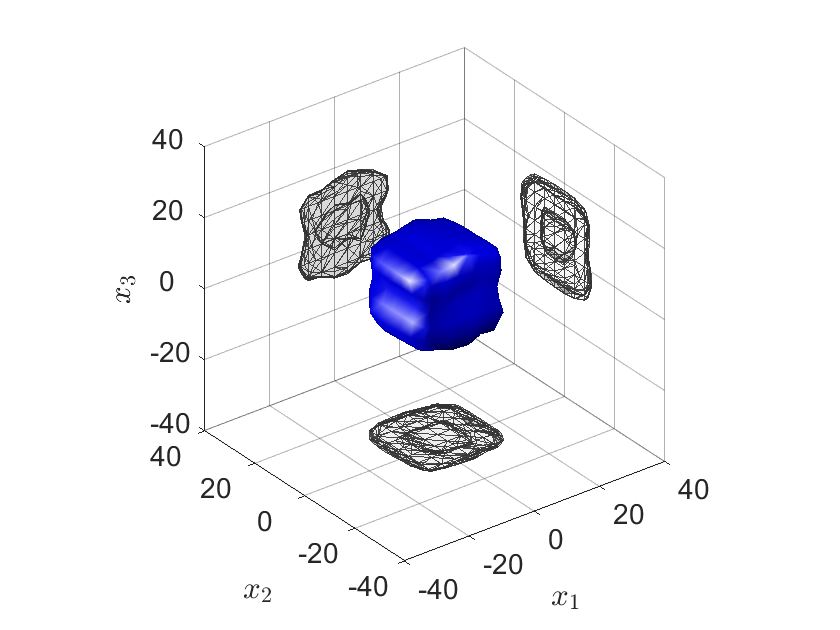} \\
            (a)& (b)& (c)\\
            \includegraphics[width=0.3\textwidth]{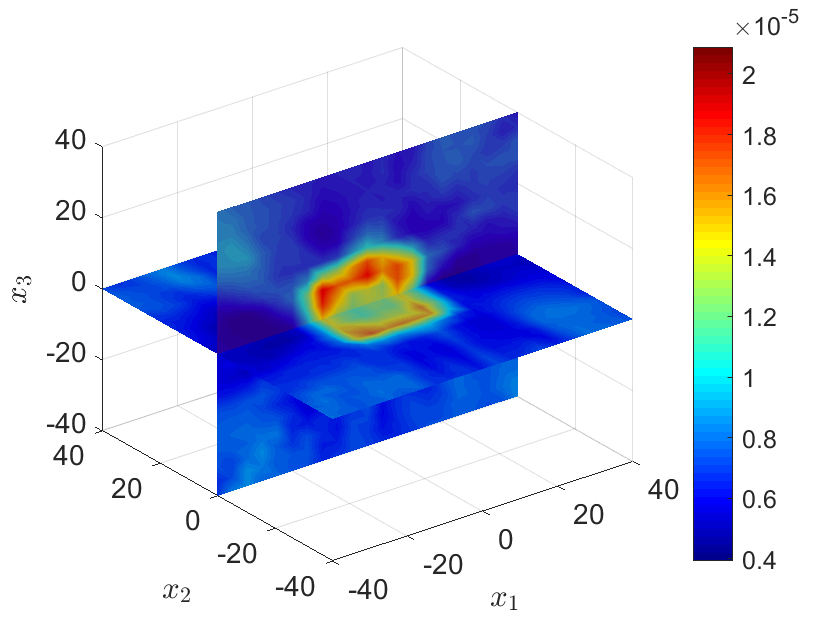}
			& \includegraphics[width=0.3\textwidth]{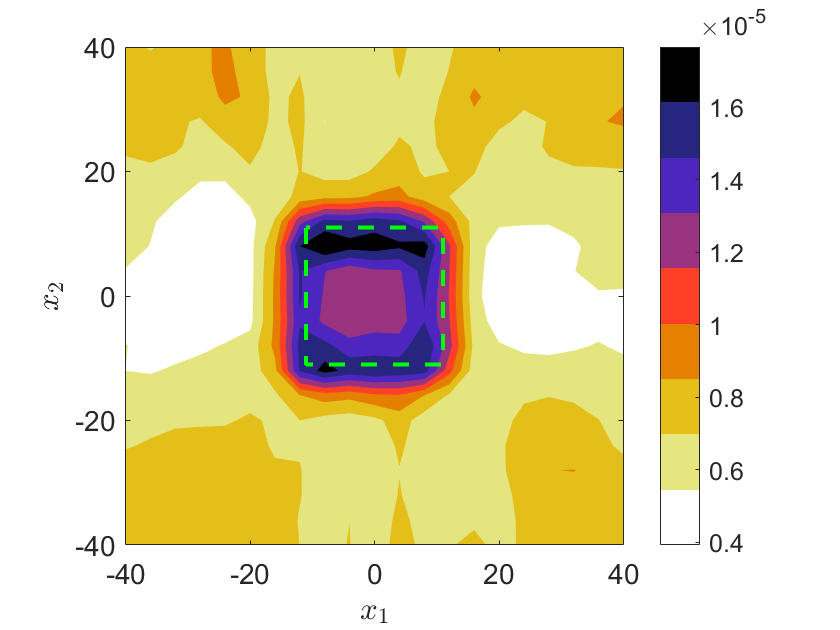}
            & \includegraphics[width=0.3\textwidth]{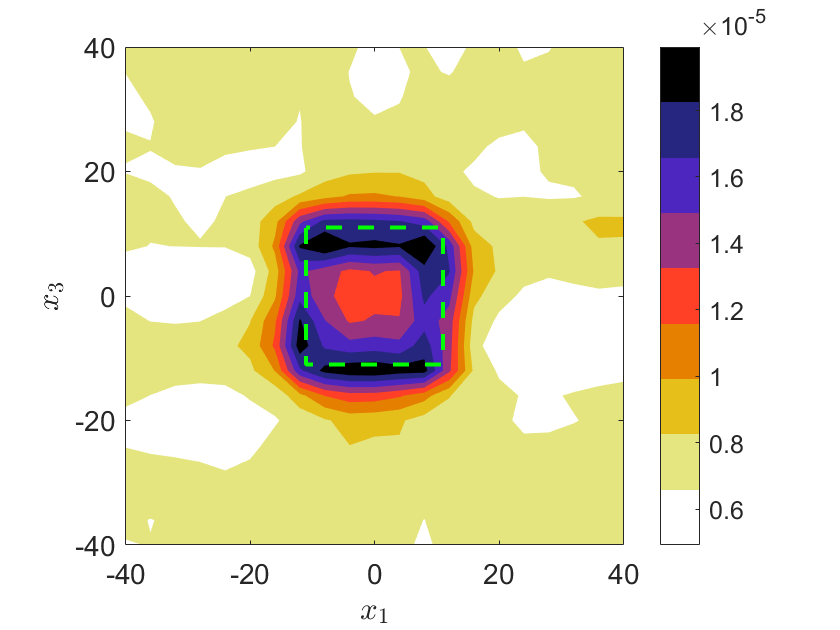} \\
            (d)& (e)& (f)\\
		\end{tabular}
\caption{The reconstruction of a single cube using the indicator function $\tilde{I}_2(z)$. The noise level is $\sigma  =5\%$. (a) Real location of the cube. (b-c) The reconstructions of the cube with the iso-surface level $1.2\times10^{-5}$ and $1.4\times10^{-5}$, respectively. (d-f) The slices through the center of the cube, where the dotted lines represent the boundaries of the cube.}
\label{1cube}
\end{figure}

\begin{figure}[!htbp]
		\centering
		\begin{tabular}{ccc}
			\includegraphics[width=0.3\textwidth]{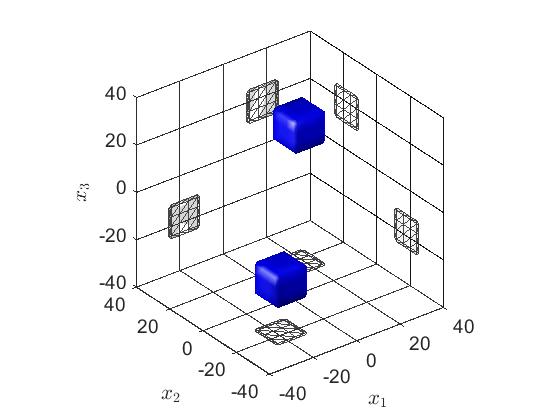}
& \includegraphics[width=0.3\textwidth]{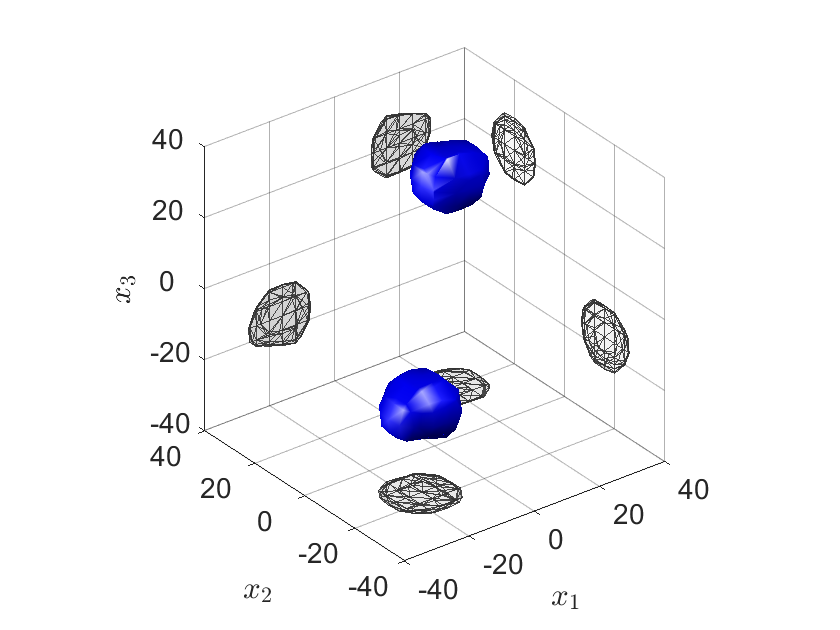}
            & \includegraphics[width=0.3\textwidth]{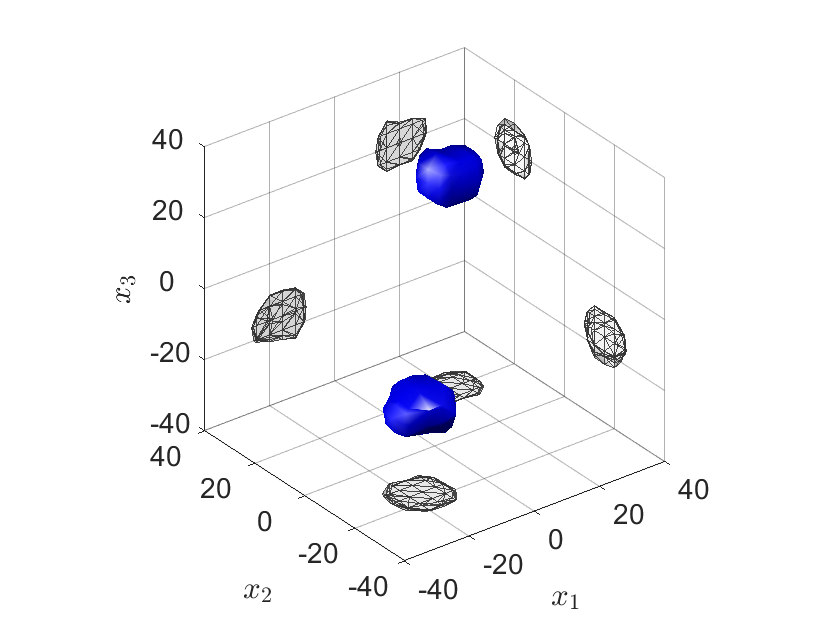} \\
            (a)& (b)& (c)\\
            \includegraphics[width=0.3\textwidth]{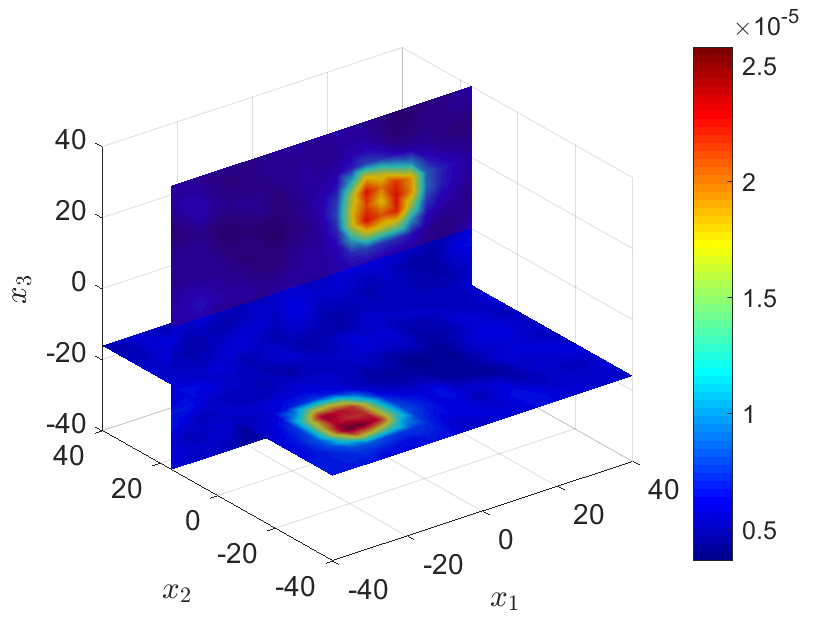}
			& \includegraphics[width=0.3\textwidth]{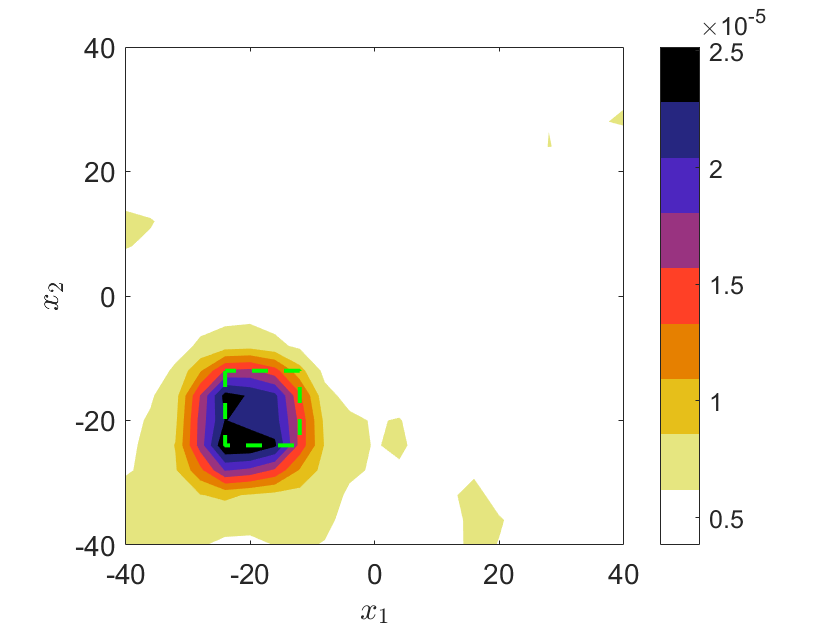}
            & \includegraphics[width=0.3\textwidth]{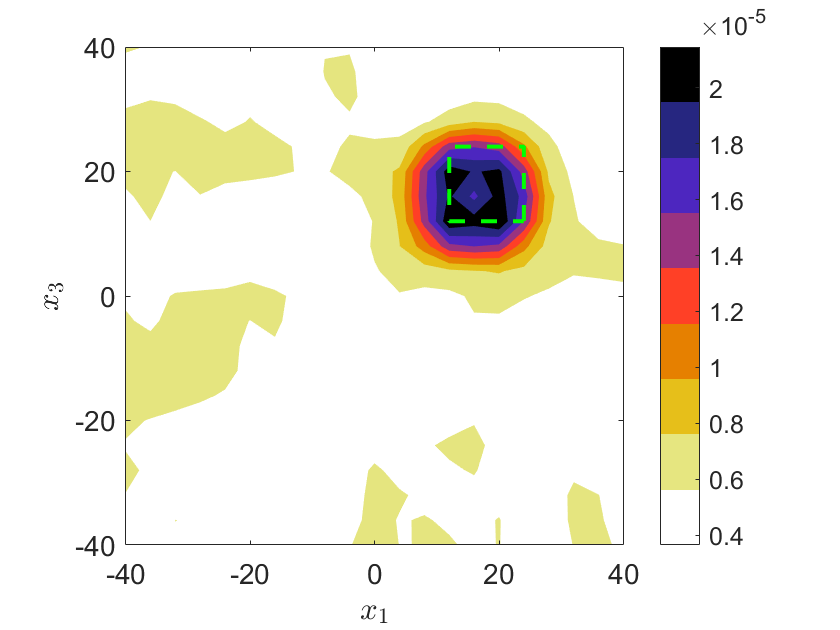} \\
            (d)& (e)& (f)\\
		\end{tabular}
\caption{The reconstruction of two disconnected cubes using the indicator function $\tilde{I}_2(z)$. The noise level is $\sigma  =5\%$.  (a) Real location of the cubes. (b-c) The reconstructions of the cubes with the iso-surface level $1.5\times10^{-5}$ and $1.8\times10^{-5}$, respectively. (d-f) The slices through the centers of the cubes, where the dotted lines represent the boundaries of the cubes.}
\label{2cubes}
\end{figure}

\section{ conclusion}
This paper has studied the time domain forward scattering and inverse scattering problems with a single moving emitter. Approximate solutions for the forward scattering problem have been analyzed. Apart from the basic idea to construct an indicator function based on the approximate solutions, a novel indicator function has been proposed to construct the direct sampling method. Comprehensive numerical experiments have been provided to show the effectiveness of the proposed algorithm for reconstructing both point-like scatterers and extended scatterers with a single moving emitter, in both two-dimensional and three-dimensional spaces.

\begin{center}
{\bf ACKNOWLEDGMENTS}
\end{center}

The work of Bo Chen was supported by the National Natural Science Foundation of China (NSFC, No. 11671170). The work of Peng Gao was supported by the Scientific Research Foundation of Civil Aviation University of China (No. 2020KYQD109).



\begin{thebibliography}{99}

\bibitem{Banjai Rapid solution}L.~Banjai and S.~Sauter, {\em Rapid solution of the wave equation in unbounded domains}, SIAM J. Numer. Anal. \textbf{47}(1), 227--249 (2008).

\bibitem{BG scattering elastic shell medium}G.~Bao and L.~Zhang, {\em The acoustic scattering of a layered elastic shell medium}, Ann. Appl. Math. \textbf{39}(4), 462--492 (2023).

\bibitem{Gang Bao iteration method}G.~Bao, P.~Li, J.~Lin, and F.~Triki, {\em Inverse scattering problems with multi-frequencies}, Inverse Probl. \textbf{31}(9), 093001 (2015). 

\bibitem{G. Bao Inverse scattering direct imaging algorithm}G.~Bao, S.~Hou, and P.~Li, {\em Inverse scattering by a continuation method with initial guesses from a direct imaging algorithm}, J. Comput. Phys. \textbf{227}(1), 755--762 (2007). 

\bibitem{Y. Chang Simultaneous recovery near-field}Y.~Chang and Y.~Guo, {\em Simultaneous recovery of an obstacle and its excitation sources from near-field scattering data}, Electron. Res. Arch. \textbf{30}(4), 1296--1321 (2022). 

\bibitem{CB scattering and inverse scattering}B.~Chen, F.~Ma, and Y.~Guo, {\em Time domain scattering and inverse scattering problems in a locally perturbed half-plane}, Appl. Anal. \textbf{96}(8), 1303--1325 (2017).

\bibitem{CB point-like scatterer}B.~Chen and Y.~Sun, {\em A simple method of reconstructing a point-like scatterer according to time-dependent acoustic wave propagation}, Inverse Probl. Sci. Eng. \textbf{29}(2), 1895--1911 (2021).

\bibitem{CB reconstruct three dimensional time-dependent point sources}	B.~Chen, Y.~Guo, F.~Ma, and Y.~Sun, {\em Numerical schemes to reconstruct three-dimensional time-dependent point sources of acoustic waves}, Inverse Probl. \textbf{36}(7), 075009 (2020).

\bibitem{Q Chen LSM effective}Q.~Chen and H.~Haddar, A.~Lechleiter and P.~Monk, {\em A sampling method for inverse scattering in the time domain}, Inverse Probl. \textbf{26}(8), 085001 (2010).

\bibitem{D Colton book}	D.~Colton and R.~Kress, {\em Inverse Acoustic and Electromagnetic Scattering Theory}, 3rd ed. Berlin: Springer (2013).

\bibitem{Integral Equation Methods}D.~Colton and R.~Kress, {\em Integral Equation Methods in Scattering Theory}, Pure and Applied Mathematics, A Wiley-Interscience Publication, New York: John Wiley \& Sons, Inc. (1983).

\bibitem{D Colton LSM inverse scattering}D.~Colton and A.~Kirsch, {\em A simple method for solving inverse scattering problems in the resonance region}, Inverse Probl. \textbf{12}(4), 383 (1996).

\bibitem{Hong Guo DSM}H.~Guo, J.~Huang, and Z.~Li, {\em Direct sampling method for solving the inverse acoustic wave scattering problems in the time domain}, Comput. Math. Appl. \textbf{179}, 229--242 (2025).

\bibitem{GYK Toward linear sampling method}Y.~Guo, P.~Monk, and D.~Colton, {\em Toward a time domain approach to the linear sampling method}, Inverse Probl.  \textbf{29}(9), 095016 (2013).

\bibitem{GYK DSM reverse time}Y.~Guo, D.~H\"{o}mberg, G.~Hu, J.~Li, and H.~Liu, {\em A time domain sampling method for inverse acoustic scattering problems}, J. Comput. Phys. \textbf{314}, 647--660 (2016). 

\bibitem{GYK DSM connection}Y.~Guo, H.~Li, and X.~Wang, {\em A novel time-domain direct sampling approach for inverse scattering problems in acoustics}, SIAM J. Appl. Math. \textbf{84}(5), 2152--2174 (2024).

\bibitem{Haddar improved time domain linear sampling for Robin and Neumann}H.~Haddar, A.~Lechleiter, and S.~Marmorat, {\em An improved time domain linear sampling method for Robin and Neumann obstacles}, Appl. Anal. \textbf{93}(2), 369--390 (2014).

\bibitem{JUN ZOU near field}K.~Ito,  B.~Jin, and J.~Zou, {\em A direct sampling method to an inverse medium scattering problem}, Inverse Probl. \textbf{28}(2), 025003 (2012).

\bibitem{J. Lai direct imaging method for the inverse obstacle scattering}J.~Lai, M.~Li, P.~Li, and W.~Li, {\em A fast direct imaging method for the inverse obstacle scattering problem with nonlinear point scatterers}, Inverse Probl. Imaging. \textbf{12}(3), 635--665 (2018).

\bibitem{K. H. Leem factorization methods inverse scattering}K. H.~Leem, J.~Liu, and G.~Pelekanos, {\em Two direct factorization methods for inverse scattering problems}, Inverse Probl. \textbf{34}(12), 125004 (2018). 

\bibitem{J. Li singular and hypersingular}J.~Li, H.~Liu, H.~Sun, and J.~Zou, {\em Imaging acoustic obstacles by singular and hypersingular point sources}, Inverse Probl. Imaging. \textbf{7}(2), 545--563 (2013).

\bibitem{X. Liu sampling method for multiple multiscale targets}X.~Liu, {\em A novel sampling method for multiple multiscale targets from scattering amplitudes at a fixed frequency}, Inverse Probl. \textbf{33}(8), 085011 (2017).

\bibitem{GYK statistical reconstruct moving sources}Y.~Liu, Y.~Guo, and J.~Sun, {\em A deterministic-statistical approach to reconstruct moving sources using sparse partial data}, Inverse Probl. \textbf{37}(6), 065005 (2021).

\bibitem{algebraic moving ui}E.~Nakaguchi, H.~Inui, and K.~Ohnaka, {\em An algebraic reconstruction of a moving point source for a scalar wave equation}, Inverse Probl. \textbf{28}(6), 065018 (2012).

\bibitem{Sayas Time Domain Boundary Integral Equations}F. J.~Sayas, {\em Retarded Potentials and Time Domain Boundary Integral Equations: A Road-map}. Switzerland: Springer (2016).

\bibitem{Yining Shen Geoacoustic inversion}
Y.~Shen, F.~Liu, and X.~Pan, {\em Geoacoustic inversion using distributed mobile underwater acoustic sensor networks}, Acta Acust. \textbf{49}(4), 774--783 (2024).

\bibitem{Y Sun Indirect boundary integral equation method}Y.~Sun, {\em Indirect boundary integral equation method for the Cauchy problem of the Laplace equation}, J. Sci. Comput. \textbf{71}(2), 469–-498 (2017).

\bibitem{Doppler Tomography Imaging} T.~Swietlik and K. J.~Opielinski, {\em Analysis for improvement of Doppler tomography imaging of objects scattering continuous ultrasonic waves}, Arch. Acoust. \textbf{45}(2), 329--339 (2020).

\bibitem{k-wave}B. E.~Treeby and B. T.~Cox, {\em K-wave: Matlab toolbox for the simulation and reconstruction of photoacoustic wave-fields}, J. Biomed. Opt. \textbf{15}, 021314 (2010).

\bibitem{WXC}X.~Wang, Y.~Guo, J.~Li, and H.~Liu, {\em Mathematical design of a novel input/instruction device using a moving emitter}, Inverse Probl. \textbf{33}(10), 105009 (2017).

\bibitem{YQQ DSM Green function}Q.~Yu, B.~Chen, J.~Wang, and Y.~Sun, {\em A direct sampling method based on the Green’s function for time-dependent inverse scattering problems}, East Asian J. Appl. Math. \textbf{15}(1), 205--224 (2025).

\bibitem{Y Yue linear sampling method for cracks}Y.~Yue, F.~Ma, and B.~Chen, {\em Time domain linear sampling method for inverse scattering problems with cracks}, East Asian J. Appl. Math. \textbf{12}(1), 96--110 (2022).

\bibitem{D. Zhang Co-inversion}D.~Zhang, Y.~Guo, Y.~Wang, and Y.~Chang, {\em Co-inversion of a scattering cavity and its internal sources: uniqueness, decoupling and imaging}, Inverse Probl. \textbf{39}(6), 065004 (2023).


\end{thebibliography}
\end{document}